\documentclass[]{amsart} %
\usepackage{amssymb, mathdots}
\usepackage{mathabx}
\usepackage{mathrsfs, hhline}
\usepackage[utf8]{inputenc}
\usepackage{amsmath,  graphicx, wasysym}
\usepackage{hyperref}
\usepackage{caption}
\usepackage{mathtools}
\usepackage{tikz}
\usetikzlibrary{matrix, arrows.meta}
\usetikzlibrary{cd}
\usepackage{wrapfig}  
\usepackage{enumerate}
\DeclareSymbolFont{bbold}{U}{bbold}{m}{n}
\DeclareSymbolFontAlphabet{\mathbbold}{bbold}
\def\qmod#1#2{{\hbox{}^{\displaystyle{#1}}}\!\big/\!\hbox{}_{
\displaystyle{#2}}}

\def\resto#1#2{{
#1\hskip 0.4ex\vline_{\hskip 0.2ex\raisebox{-0,2ex}
{{${\scriptstyle #2}$}}}}}

\def\C{{\mathbb C}}

\def\N{{\mathbb N}}

\def\P{{\mathbb P}}
\def\Q{{\mathbb Q}}

\def\Z{{\mathbb Z}}

\def\union{\mathop{\bigcup}}

\def\textmap#1{\mathop{\vbox{\ialign{
                                  ##\crcr
      ${\scriptstyle\hfil\;\;#1\;\;\hfil}$\crcr
      \noalign{\kern 1pt\nointerlineskip}
      \rightarrowfill\crcr}}\;}}
\def\bigtextmap#1{\mathop{\vbox{\ialign{
                                  ##\crcr
      ${\hfil\;\;#1\;\;\hfil}$\crcr
      \noalign{\kern 1pt\nointerlineskip}
      \rightarrowfill\crcr}}\;}}
      
\newcommand{\cal}{\mathcal}
\def\textlmap#1{\mathop{\vbox{\ialign{
                                  ##\crcr
      ${\scriptstyle\hfil\;\;#1\;\;\hfil}$\crcr
      \noalign{\kern-1pt\nointerlineskip}
      \leftarrowfill\crcr}}\;}}

\def\g{{\mathfrak g}}

\def\mg{{\mathfrak m}}

\def\sg{{\mathfrak s}}

\def\Dg{{\mathfrak D}}

\def\Wg{{\mathfrak W}}

\theoremstyle{remark}
\newtheorem{ex}{Example}[section]

\theoremstyle{plain}
\newtheorem{sz}{Satz}[section]
\newtheorem{thry}[sz]{Theorem}
\newtheorem{pr}[sz]{Proposition}
\newtheorem{re}[sz]{Remark}
\newtheorem{co}[sz]{Corollary}
\newtheorem{dt}[sz]{Definition}
\newtheorem{lm}[sz]{Lemma}

\def\tr{\mathrm {Tr}}
\def\End{\mathrm {End}}

\def\SO{\mathrm {SO}}

\def\GL{\mathrm {GL}}

\def\gl{\mathrm {gl}}

\def\Hom{\mathrm{Hom}}

\def\id{ \mathrm{id}}
\def\im{\mathrm{im}}
\def\rk{\mathrm {rk}}
\def\ad{\mathrm {ad}}

\def\st{\mathrm {st}}

\def\ss{\mathrm {ss}}

\def\us{\mathrm {us}}

\def\oo{{\scriptstyle{\cal O}}}

\def\Crit{\mathrm{Crit}}
\def\Qcoh{\mathrm{Qcoh}}
\newcommand{\catqot}{/\hskip-2pt/}
\newcommand\smvee{{\hskip -0.15ex \raise 0.2ex\hbox{$\scriptscriptstyle\vee$}\hskip -0.3ex}}
\newcommand{\extpw}{\mathchoice{{\textstyle\bigwedge}}%
    {{\bigwedge}}%
    {{\textstyle\wedge}}%
    {{\scriptstyle\wedge}}}
\def\extp{{\extpw}\hspace{-2pt}}

\def\Spec{\mathrm{Spec}}
\def\Ext{\mathrm{Ext}}

\def\Gr{\mathrm{Gr}}
\def\coh{\mathrm{coh}}

\def\Qcoh{\mathrm{Qcoh}}
\def\Perf{\mathrm{Perf}}
\def\Mod{\mathrm{Mod}}
\def\mod{\mathrm{mod}}

\def\gr{\mathrm{gr}}
\def\mod{\mathrm{mod}}
\def\Mod{\mathrm{Mod}}
\def\op{{\mathrm{op}}}

\makeatletter
\DeclareFontFamily{OMX}{MnSymbolE}{}
\DeclareSymbolFont{MnLargeSymbols}{OMX}{MnSymbolE}{m}{n}
\SetSymbolFont{MnLargeSymbols}{bold}{OMX}{MnSymbolE}{b}{n}
\DeclareFontShape{OMX}{MnSymbolE}{m}{n}{
    <-6>  MnSymbolE5
   <6-7>  MnSymbolE6
   <7-8>  MnSymbolE7
   <8-9>  MnSymbolE8
   <9-10> MnSymbolE9 
  <10-12> MnSymbolE10
  <12->   MnSymbolE12
}{}
\DeclareFontShape{OMX}{MnSymbolE}{b}{n}{
    <-6>  MnSymbolE-Bold5
   <6-7>  MnSymbolE-Bold6
   <7-8>  MnSymbolE-Bold7
   <8-9>  MnSymbolE-Bold8
   <9-10> MnSymbolE-Bold9
  <10-12> MnSymbolE-Bold10
  <12->   MnSymbolE-Bold12
}{}

\let\llangle\@undefined
\let\rrangle\@undefined
\DeclareMathDelimiter{\llangle}{\mathopen}%
                     {MnLargeSymbols}{'164}{MnLargeSymbols}{'164}
\DeclareMathDelimiter{\rrangle}{\mathclose}%
                     {MnLargeSymbols}{'171}{MnLargeSymbols}{'171}
\makeatother
\newcommand{\vertiii}[1]{{\left\vert\kern-0.25ex\left\vert\kern-0.25ex\left\vert #1 
    \right\vert\kern-0.25ex\right\vert\kern-0.25ex\right\vert}} 

\begin{document} 

\title[Graded tilting for gauged Landau-Ginzburg models]{Graded tilting for gauged Landau-Ginzburg models and geometric applications}
\author{Christian Okonek \and Andrei Teleman}
\address{Christian Okonek: Institut für Mathematik, Universität Zürich,
Winterthurerstrasse 190, CH-8057 Zürich, Switzerland,  e-mail: okonek@math.uzh.ch}
\address{Andrei Teleman: Aix Marseille Université, CNRS, Centrale Marseille, I2M, UMR 7373, 13453 Marseille, France, email: andrei.teleman@univ-amu.fr}

\begin{abstract}

In this paper we develop a graded tilting theory for gauged Lan\-dau-Ginzburg models of regular sections in vector bundles over projective varieties. Our main theoretical result describes -- under certain conditions -- the bounded derived category of the zero locus $Z(s)$ of such a section $s$ as a   graded singularity category of a non-commutative quotient algebra $\Lambda/\langle s\rangle: D^b(\coh Z(s))\simeq D^{\gr}_\mathrm{sg}(\Lambda/\langle s\rangle)$. Our geometric applications all come from homogeneous gauged linear sigma models. In this case $\Lambda$ is a graded non-commutative resolution of the invariant ring which defines the $\C^*$-equivariant affine GIT quotient of the model. 

We obtain  algebraic descriptions of the derived categories of the following families of varieties: 
\begin{enumerate}[1.]
\item Complete intersections.
\item Isotropic symplectic and orthogonal Grassmannians.	
 \item Beauville-Donagi IHS 4-folds.
\end{enumerate}

\end{abstract}

\maketitle

\tableofcontents

\section{Introduction}

\subsection{Motivation}

The bounded derived category of coherent sheaves $D^b(\coh Z)$ of a smooth projective variety $Z$ is one of the most important invariants. It determines the variety up to isomorphism when the (anti-)canonical line bundle is ample  \cite{BO}. It is expected to play a fundamental role in the minimal model program and in connection with the homological mirror symmetry conjecture \cite{Or2}. There are several methods available for describing $D^b(\coh Z)$, some of which can (or are expected to) work only for special classes. The first method uses  full exceptional collections (conjectured by Orlov to work only for rational varieties), or more generally non-trivial semi-orthogonal decompositions (which do not exist when the canonical line bundle is trivial \cite{KO}). There is homological projective duality, which is a very powerful method to construct  semi-orthogonal decompositions, but where examples are hard to find \cite{Ku}. Then we have variation of GIT quotients, which works for a variety $Z$ which can be identified with a GIT quotient $U\catqot_{\chi_+}G$, but needs restrictive symmetry properties of the Kempf-Ness stratifications associated with two linearizations $\chi_\pm$ of the $G$-action on $U$ \cite{BFK}, \cite{HLS}. On the other hand equivalences between derived categories of smooth projective varieties can -- at least in principle --  always be described by suitable Fourier-Mukai functors \cite{Or3}. Finally there is  geometric tilting theory \cite{HVdB}, the method we will use in this paper.

Many interesting varieties are defined as the zero locus $Z(s)$ of a section $s\in H^0(B,E^\smvee)$, where $B$ is a smooth projective variety, $E^\smvee\to B$ is a globally generated vector bundle, and $s$ is general.

 We will develop a graded geometric tilting theory applied to the gauged Landau-Ginzburg model $s^\smvee:E\to \C$, where $E$ and $\C$ are endowed with the natural $\C^*$-actions. Here $s^\smvee$ is the potential associated with the section $s$,   defined by $s^\smvee(e)=\langle s(b), e\rangle$ for $e\in E_b$.

The starting point of this approach is the Isik-Shipman theorem, which gives equivalences
$$D^b(\coh(Z(s))\simeq D^\gr_{\mathrm{sg}} (Z(s^\smvee))\simeq D(\coh_{\C^*}E,\theta,s^\smvee).
$$
Here $ D^\gr_{\mathrm{sg}} (Z(s^\smvee))$ stands for the graded singularity category of the zero locus $Z(s^\smvee)$ of the potential, $\theta$ is the character $\id_{\C^*}$,   and $D(\coh_{\C^*}E,\theta,s^\smvee)$ denotes the derived factorization category associated with the 4-tuple $(\C^*,E,\theta,s^\smvee)$ \cite{Hi1}.

\subsection{Results} Let $B$ be a smooth projective variety, $H$ a finite dimensional complex vector space, and $E$ a subbundle of the trivial vector bundle $B\times H^\smvee$.  Under these assumptions $E$   fits in a commutative diagram
\begin{equation}\label{DiagE}
\begin{tikzcd}
E  \ar[d, two heads, "\rho" '] \ar[r, hook] \ar[dr, "\vartheta"] & B\times H^\smvee  \ar[d, two heads]\\
C(E) \ar[r, hook] & H^\smvee	
\end{tikzcd}
\end{equation}
where $\vartheta$ is projective, $C(E)\coloneqq\im(\vartheta)$ is a subvariety of the affine space $H^\smvee$, and $\rho$ is induced by $\vartheta$. $C(E)$ coincides with the affine cone of the projective variety $S(E)\coloneqq \im(\P(\vartheta))\subset\P(H^\smvee).$ We will assume
 \vspace{2mm}\\
(A1) ${\cal T}_0$ is a locally free sheaf generating $D(\Qcoh (B))$, and ${\cal T}\coloneqq \pi^*({\cal T}_0)$ is a tilting sheaf on $E$.
  \vspace{2mm}

Define 
$$
\Lambda\coloneqq  \End_E({\cal T}),\ R\coloneqq \C[C(E)].
$$
Then
\begin{enumerate}[(i)]
\item $\Lambda$ is a graded Noetherian $\C$-algebra of finite graded global dimension, and the graded tilting functor
$$R\Hom^\gr_{\Qcoh_{\C^*} E} ({\cal T},-):  D(\Qcoh_{\C^*} E)\to D(\Mod_\Z\Lambda)
$$
is an equivalence, which restricts to an equivalence
$$D^b(\coh_{\C^*}E)\textmap{\simeq} D^b(\mod_\Z\Lambda).
$$
\item A regular section $s\in H^0(B,E^\smvee)$ defines a central element $s\in \Lambda$, and ${\cal T}$ defines an equivalence
$${\cal T}_*: D^\gr_\mathrm{sg}(Z(s^\smvee))\textmap{\simeq} D^\gr_\mathrm{sg}(\Lambda / s\Lambda).
$$
Combining this result with the Isik-Shipman theorem, one obtains an equivalence
$$D^b(\coh(Z(s))\textmap{\simeq} D^\gr_\mathrm{sg}(\Lambda / s\Lambda)
$$
which gives a tilting description of the  category $D^b(\coh(Z(s))$.
\item If we also assume 
\vspace{2mm}\\ 
\hspace*{-10mm}
(A2) $B$ is connected, the cone $C(E)\subset H^\smvee	$ is normal, and  $\rho:E\to C(E)$ is birational with $\mathrm{codim}_E\big(\rho^{-1}(\mathrm{Exc}(\rho))\big)\geq 2$, 
\vspace{2mm}\\ 
then $R=\bigoplus_{m\geq 0} H^0(B,S^m {\cal E}^\smvee) $, and  $\Lambda\supset R$ is a graded non-commutative resolution of $R$. In several cases   this resolution is crepant.	
\end{enumerate}
\vspace{2mm} 

The proof of part (i) is   non-trivial: it uses a generalized Beilinson lemma (Lemma \ref{Schwede}), and an infinite family of generators for $D^b(\coh_{\C^*} E)$ (Theorem \ref{prop-old-l2-2}).  Part (ii) is stated in Theorem 
\ref{FourthEq}, which  can be understood as a Baranovsky-Pecharich-type result \cite{BaPe} in the context of tilting theory. The proof   uses the Isik-Shipman theorem and  ideas of \cite[Theorem 5.1]{BDFIK} to identify $D^b(\coh(Z(s))$ with a homotopy category of graded matrix factorizations $K(\mathrm{proj}_\Z\Lambda,\langle 1\rangle,s)$. In a second step we prove a version of Orlov's comparison theorem \cite[Theorem 3.10]{Or} to identify $K(\mathrm{proj}_\Z\Lambda,\langle 1\rangle,s)$ with the triangulated graded singularity category $D^\mathrm{gr}_\mathrm{sg}(\Lambda/\langle s\rangle)$. Our version of the comparison theorem is necessary since we have to deal with algebras $\Lambda$ which are not connected.
\vspace{2mm}

We will construct and study a large class of   diagrams (\ref{DiagE}) satisfying   (A1)-(A2) using geometric objects which we call (inspired by the terminology of the physicists)  homogeneous algebraic geometric GLSM presentations. In our formalism a  GLSM presentation of $(E\stackrel{\pi}{\to}B,s)$ is a 4-tuple $(G,U,F,\chi)$ where  $G$ is a reductive Lie group, $U$, $F$ are finite dimensional $G$-representations, and $\chi:G\to \C^*$ is a character, such that the following conditions are satisfied (see Definition \ref{GLSMDef}):
\begin{enumerate} 
\item The stable locus $U_\st^\chi$ coincides with the semistable locus $U_{\ss}^\chi$,  and  $G$ acts freely	 on $U_{\ss}^\chi$.
\item The base $B$ coincides with the quotient $U_{\ss}^\chi/G$, and the vector bundle $E$ is the   $F$-bundle associated with the principal $G$-bundle $p:U_{\ss}^\chi\to  B$ and the representation  space $F$.
\item   The $G$-equivariant map $\hat s:U_{\ss}^\chi\to F^\smvee$ corresponding to $s$ extends to a $G$-equivariant  polynomial map  $\sigma:U\to F^\smvee$.
\item  Any optimal destabilizing 1-parameter subgroup $\xi$ of an unstable  point $u\in U^\chi_\us$ acts with non-negative weights on $F$.
\end{enumerate}

In this situation a section $s\in H^0(B,E^\smvee)$ is induced by a covariant $\sigma: U\to F^\smvee$. The fourth condition implies $ U^\chi_\ss\times F=(U\times F)^\chi_\ss$, so the bundle $E=(U^\chi_\ss\times F)/G$  is the GIT quotient associated with  the $G$-representation $W\coloneqq U\oplus F$  and the character $\chi$.

A GLSM presentation is homogeneous if there is right action $U\times {\cal G}\to U$ by linear isomorphisms of a reductive group  ${\cal G}$ which commutes with the fixed $G$-action, such that the induced ${\cal G}$-action  on $B$ is transitive and, for a point $u_0\in U^\chi_\ss$,  the obvious morphism ${\cal G}_{[u_0]}\to G$ is surjective (see section \ref{HomGLSMSect}). If this holds,  the map
$$\rho: E\to C(E)
$$ 
is a Kempf collapsing, hence $C(E)$ is a normal Cohen-Macaulay variety. The map $\rho$ is birational iff $\dim(E)=\dim(C(E))$, and if this is the case,  $C(E)$  has rational singularities.

\subsection{Applications}

Consider the $G$ representation space $W\coloneqq U\oplus F$ associated with a homogeneous GLSM.  Then 
\begin{enumerate}
\item $S(E)$ is projectively normal and arithmetically Cohen-Macaulay. If $\rho$ is birational, its affine cone $C(E)$ has rational singularities (see Corollary \ref{PropertiesS(E)} (i)).
\item Suppose that $\mathrm{codim}(U_\us^\chi)\geq 2$ and $\rho$ is birational. Then the cone $C(E)$ is naturally identified with the GIT quotient $\Spec(\C[W]^G)$ (see Corollary \ref{PropertiesC(E)} (i)).
\end{enumerate}

Moreover, we give explicit criteria (see Lemma \ref{GorCritLm}) for testing, for a given homogeneous GLSM presentation, if the cone $C(E)$  (the projective variety $S(E)$) is (arithmetically) Gorenstein, and we check these criteria in several situations.

Next we identify an important class of GLSM presentations  to which our general results apply. This is the class of 4-tuples $(G,U,F,\chi)$ with $G=\GL(Z)$ for a $k$-dimensional complex vector space $Z$, $U=\Hom(V,Z)$ with $V$ a complex vector space of dimension $N$ and $F^\smvee$ a finite dimensional polynomial representation of $\GL(Z)$. The character $\chi$ is chosen such that $U^\chi_\ss/\GL(Z)$ becomes the Grassmannian $\Gr_k(V^\smvee)$. In many cases it is possible to go one step further, and to give a purely algebraic description of the quotient $\Lambda/ s\Lambda$ in terms of the initial data  $(V,k,\lambda,\sigma)$. In order to do this, we need section 1.5 to get the identification $\Lambda=\End_R(M)$ (see section \ref{sec:2-4-2}).

   
   Then we identify $R$ with a quotient $S^\bullet(S^\lambda V)/I_k$ following Porras \cite{Po}, and we describe  $M$ as the image of a morphism between free graded $R$-modules (see Proposition \ref{prop:2-13}).   When certain conditions are satisfied, our final result takes the following form (Theorem  \ref{th:2-14}):
$$D^b(\coh Z(s))\simeq D^{\mathrm{gr}}_{\mathrm{sg}}\big(\End_{S/I_k}\big(\bigoplus_{\alpha\in P(k,n-k)}\mathrm{Im}(S^\alpha(\varphi^\smvee)\otimes (S/I_k))\big)\big/\langle\bar\sigma\rangle\big).
$$
This gives a purely algebraic description of $D^b(\coh Z(s))$ in terms of the initial data $(V,k,\lambda,\sigma)$.

\subsection{Examples}

We apply our general formalism and our results to the gauged Landau-Ginzburg models associated with the following varieties (described as zero loci of regular sections):
\begin{enumerate}[1.]
\item 	Complete intersections. In this case $B=\P(V^\smvee)$ (for a complex vector space of dimension $N>1$), $E$ is the total space of the direct sum $\bigoplus_{i=1}^r {\cal O}(-d_i)$ (where $d_i>0$) and $s$ is a general element in 
$$H^0(\bigoplus_{i=1}^r {\cal O}_{V^\smvee}(d_i))=\bigoplus_{i=1}^r S^{d_i}V.$$
\item Isotropic Grassmannians.

Let $V$ be a complex vector space of even dimension $N=2n$, and $k$ a positive even integer with $k\leq n$. Let $\omega\in \extp^2V$ be a symplectic form on $V^\smvee$. The isotropic Grassmannian $\Gr_k^\omega(V^\smvee)\subset \Gr_k(V^\smvee)$ is the submanifold of $k$-dimensional isotropic subspaces of $(V^\smvee,\omega)$. 
Denoting by $T$ the tautological $k$-bundle of $\Gr_k(V^\smvee)$, the form $\omega$ defines a section $s_\omega\in \Gamma(\Gr_k(V^\smvee),\extp^2 T^\smvee)$ which is transversal to the zero section, and whose zero locus is  $\Gr_k^\omega(V^\smvee)$.  %

Similarly, let $q\in S^2V$ be a non-degenerate quadratic form on $V^\smvee$, and choose $k \leq N/2$. The isotropic Grassmannian $\Gr_k^q(V^\smvee)\subset \Gr_k(V^\smvee)$ is the submanifold of $k$-dimensional isotropic subspaces of $(V^\smvee,q)$. 
The form $q$ defines a section $s_q\in \Gamma(\Gr_k(V^\smvee),S^2 T^\smvee)$ which is transversal to the zero section, and whose zero locus is  $\Gr_k^q(V^\smvee)$. 
\vspace{2mm}

\item  Beauville-Donagi IHS 4-folds.

With the same notation as above we take $E=S^3 T$ on the Grassmannian $\Gr_2(V^\smvee)$. The Beauville-Donagi IHS 4-folds are obtained in the special case  $\dim(V)=6$.
\end{enumerate}
\vspace{3mm} 

In all these cases the tilting bundle ${\cal T}_0$ is the direct sum of the sheaves of a full strongly exceptional collection.  In the first case  we choose the standard Beilinson collection, and in the other cases (when the basis is a Grassmannian) we use the Kapranov collection.

\subsection{Acknowledgments}

This article  builds on and combines fundamental contributions of many mathematicians, notably of Ballard - Favero - Deliu - Isik - Katzarkov \cite{BDFIK}, Bondal - Orlov \cite{BO}, Buchweitz - Leuschke - Van den Bergh  \cite{BLVdB2}, Kempf \cite{Ke}, and Orlov \cite{Or}.

We are very grateful to Alexei Bondal for his important remarks at the beginning of this project. We also thank Andrew Kresch   for his interest and pointers to the literature, and Greg Stevenson for a useful e-mail exchange.

\newpage

\section{Graded tilting for gauged Landau-Ginzburg  models}\label{AlgSection}

\subsection{The Landau-Ginzburg model of a section}\label{LGmodels}

Let $B$ be a smooth complex variety, $\pi:E\to B$  a  rank $r$ vector bundle on $B$, and $s\in\Gamma(E^\smvee)$  a  section in its dual. The zero locus $Z(s)$  is a local complete intersection of codimension $r$ when $s$ is regular. 
If, moreover, $s$ is transversal to  the zero section then $Z(s)$ is a smooth submanifold of codimension $r$.

\begin{dt} Let $s\in\Gamma(E^\smvee)$ be a  section. The potential associated with $s$ is the map $s^\smvee:E\to\C$ defined by
$$s^\smvee(y)\coloneqq \langle s(x),y\rangle\ ,\ \forall x\in B, \forall y\in E_x.
$$
\end{dt}

\def\Crit{\mathrm{Crit}}
Let $x\in B$ and $y\in E_x$. For a suitable open neighborhood $U$ of $x$ identify  the bundles $E_U$, $E^\smvee_U$  with $U\times\C^r$,  $U\times{\C^r}^\smvee$ respectively using mutually dual  trivializations $\theta$, $\theta^\smvee$. Denote by $s_\theta:U\to{\C^r}^\smvee$ the map corresponding to $s$ via $\theta^\smvee$. For any   $y\in E_x$ and a tangent vector  $(\dot x,\dot z)\in T_y(E)=T_x(B)\times{\C^r}$ one has
$$d_y s^\smvee(\dot x,\dot z)=\langle s_\theta(x),\dot z\rangle+ \langle d_xs_{\theta}(\dot x),y\rangle.
$$

This formula shows that $\Crit(s^\smvee)\subset p^{-1}(Z(s))$. When  $x\in Z(s)$,  the differential of $s^\smvee$ at a point $y\in E_x$ can be written in an invariant way:
$$d_y s^\smvee (\dot y)=\langle D_x s(p_*(\dot y)),   y\rangle \ \forall \dot y\in T_y(E), 
$$
where $D_x s:T_x B\to E_x^\smvee$ stands for the intrinsic derivative of $s$ at $x$. This shows that the critical locus $\Crit(s^\smvee)$ of $s^\smvee$ is
\begin{equation}\label{crit}
\Crit(s^\smvee)=\union_{x\in Z(s)}  \bigg\{\qmod{E_x^\smvee}{\im(D_xs)}\bigg\}^\smvee.
\end{equation}
 In particular
\begin{re}\label{transv}
If $s$ is a transversal to the zero section, then $\Crit(s^\smvee)$ coincides as a variety with the image of $Z(s)$ via the zero section $\oo:B\to E$, so it can be identified with $Z(s)$ via $\pi$.
\end{re}
We refer to \cite{Or2} and \cite{Hi2} for the following fundamental definition:
\begin{dt} A gauged Landau-Ginzburg model is a 4-tuple $(G,X,\kappa,w)$, where $G$ is  an algebraic group, $\kappa\in\Hom(G,\C^*)$ is a character, $X$  is a  smooth $G$-variety, and $w:X\to \C$ is a $\kappa$-equivariant regular function on $X$, called the potential of the model.   

 Let  $\pi:E\to B$ be a  rank $r$ vector bundle on $B$, and $s\in\Gamma(E^\smvee)$.
The 4-tuple $(E,\C^*,\id_{\C^*}, s^\smvee)$, where $E$ is endowed  with the fibrewise scaling $\C^*$-action,  will be called the gauged Landau-Ginzburg model associated with  $(E\to B,s)$. 
\end{dt}

This class of gauged Landau-Ginzburg models will play a fundamental role in  this article.



\subsection{\texorpdfstring{$\C^*$}{str1}-equivariant derived categories of vector bundles}

Let $B$ be a smooth projective scheme, $H$ a finite dimensional $k$-vector space, and  let
$$
\begin{tikzcd}
E \ar[r, hook] \ar[dr, two heads, "\pi"'] &	B\times H^\smvee\ar[d, two heads]\\
 & B
\end{tikzcd}
$$
be a sub-vector bundle of the trivial vector bundle $B\times H^\smvee$. This implies that $E$ is projective over $H^\smvee$, in particular it belongs to the class of varieties concerned by geometric tilting (see \cite[Theorem 7.6]{HVdB}, \cite[sections 1.8, 1.9]{BH}).
%

\begin{pr} \label{prop-old-l2-1} Let ${\cal T}_0$ be a coherent sheaf on $B$ which generates $D(\Qcoh B)$. Then ${\cal T}\coloneqq \pi^*({\cal T}_0)$ generates $D(\Qcoh E)$.
 \end{pr}
 \begin{proof}
 Since $\pi$ is affine, it follows that the functor
$$\pi_*:\Qcoh E\to \Qcoh B
$$
(\cite[Lemma 25.24.1]{Stack}) is exact.  Its left adjoint functor is
$$\pi^*:\Qcoh B\to \Qcoh E
$$
(\cite[Lemmas 6.26.2, 17.10.4]{Stack}), and this functor is also exact because $\pi$ is flat (\cite[Lemma 28.11.6]{Stack}). Since the functors $\pi_*$, $\pi^*$ are exact, they induce well defined functors
$$D(\pi_*): D(\Qcoh E)\to D(\Qcoh B),\ D(\pi^*):D(\Qcoh B)\to D(\Qcoh E)
$$
which act on complexes componentwise, and are  right and left derived functors of $\pi_*$, respectively  $\pi^*$ \cite[p. 75]{BDG}.  Using  Lemma \cite[Lemma 13.28.5]{Stack} it follows that $D(\pi_*)$ is a right adjoint for $D(\pi^*)$. Therefore for any two objects ${\cal M}$, ${\cal N}$ in $D(\Qcoh E)$,  $D(\Qcoh B)$ respectively  one has an identification
\begin{equation}\label{adj}
\Hom_{D(\Qcoh E)}(D(\pi^*)({\cal M}),{\cal N})=\Hom_{D(\Qcoh B)}({\cal M},D(\pi_*)({\cal N})). 	
\end{equation}
 Let now ${\cal N}$ be a complex of quasi-coherent sheaves on $E$ such that
$$\Hom_{D(\Qcoh E)}({\cal T}, {\cal N}[m])=0\ \forall m\in\Z.
$$
By (\ref{adj}) we obtain
$$\Hom_{D(\Qcoh B)}({\cal T}_0,D(\pi_*)({\cal N})[m])=0 \ \forall m\in\Z,
$$
so that $D(\pi_*)({\cal N})=0$ in $D(\Qcoh B)$, because ${\cal T}_0$ generates $D(\Qcoh B)$. Therefore $D(\pi_*)({\cal N})$ is an acyclic complex.  Using \cite[Lemma 28.11.6]{Stack} it follows that ${\cal N}$ is an an acyclic complex, too, hence ${\cal N}=0$ in $D(\Qcoh E)$.	
 \end{proof}

 \begin{co}\label{co3}
Let ${\cal T}_0$ be a locally free coherent sheaf on $B$ which classically generates $D^b(\coh B)$. Then 
\begin{enumerate}
\item ${\cal T}_0$ generates $D(\Qcoh B)$. 	
\item 	$\pi^*({\cal T}_0)$  classically generates $D^b(\coh E)$.
\item The pull-back $\pi^*({\cal T}_0)$ is a tilting object of $D(\Qcoh  E)$ if and only if  
$$H^i\big(B,{\cal T}_0^\smvee\otimes {\cal T}_0\otimes (\bigoplus_{m\geq 0} S^m {\cal E}^\smvee)\big)=0 \ \forall i>0. 
$$

\end{enumerate}
	
\end{co}
\begin{proof} (1) $D(\Qcoh B)$ is compactly generated, and $D(\Qcoh B)^c$ is equivalent to $D^b(\coh B)$ because $B$ is smooth. The claim follows from Ravenel-Neeman's  Theorem (see \cite[Theorem 2.1.2]{BVdB}, \cite[section 1.4]{BH}). 
\\ \\
(2)
The sheaf $\pi^*({\cal T}_0)$ 	generates $D(\Qcoh E)$. Since the composition
$$ E\hookrightarrow B\times H^\smvee \to H^\smvee  
$$
is a projective morphism, $D(\Qcoh E)$ is compactly generated. But  $\pi^*({\cal T}_0)$  is an object in $D_\Perf(\Qcoh E)=D(\Qcoh E)^c$ which generates $D(\Qcoh E)$, hence it classically generates $D(\Qcoh E)^c\simeq D^b(\coh E)$ because $E$ is smooth. 
\\ \\
(3) ${\cal T}$ is a compact generator of	$D(\Qcoh E)$, and 
$$\Ext^i({\cal T},{\cal T})=H^i(E,\pi^*({\cal T}_0^\smvee\otimes {\cal T}_0))=H^i(B,{\cal T}_0^\smvee\otimes {\cal T}_0\otimes \pi_*({\cal O}_E))$$
$$=H^i\big(B,{\cal T}_0^\smvee\otimes {\cal T}_0\otimes (\bigoplus_{m\geq 0} S^m {\cal E}^\smvee)\big)=0\ \forall i>0. 
$$
\end{proof}
\begin{lm}\label{newLm}
The canonical functor $D^b(\coh_{\C^*} E)\to 	D(\Qcoh_{\C^*} E)$ defines an equivalence between $D^b(\coh_{\C^*} E)$ and the full subcategory of $D(\Qcoh_{\C^*} E)$ of complexes with bounded, coherent cohomology.
\end{lm}
\def\bdcoh{\mathrm{bd-coh}}
\def\bd{\mathrm{bd}}
\begin{proof}
The category of $\C^*$-equivariant quasi-coherent sheaves on $E$ can be identified with the category of quasi-coherent sheaves on the quotient stack $[E/\C^*]$ \cite{Tho}. Combining this identification with  \cite[Corollary 2.11 p. 10]{ArBe}   we  see  that  the canonical functor 
$$D^b(\coh_{\C^*} E)\to 	D^b(\Qcoh_{\C^*} E)$$
gives an equivalence
$$D^b(\coh_{\C^*} E)\to 	D^b_\coh(\Qcoh_{\C^*} E)
$$
with the full subcategory $D^b_\coh(\Qcoh_{\C^*} E)$ of $	D^b(\Qcoh_{\C^*} E)$ consisting of objects  with coherent cohomology. 

On the other hand, \cite[Lemma 11.7, p. 15]{Kell} gives an equivalence
$$U:D^b(\Qcoh_{\C^*} E)\to D_{\bd}(\Qcoh_{\C^*} E)
$$
 from $D^b(\Qcoh_{\C^*} E)$ to the full subcategory $D_{\bd}(\Qcoh_{\C^*} E)$ of $D(\Qcoh_{\C^*} E)$ consisting of objects with bounded cohomology.
 
 Since cohomology is invariant under isomorphisms in the derived category,  this equivalence restricts to an equivalence $D^b_\coh(\Qcoh_{\C^*} E)\to D_{\bdcoh}(\Qcoh_{\C^*} E)$ between $D^b_\coh(\Qcoh_{\C^*} E)$ and the full subcategory of $D_{\bd}(\Qcoh_{\C^*} E)$ consisting of objects with coherent cohomology:
 $$
 \begin{tikzcd}[column sep=6mm]
 & D_\bdcoh	(\Qcoh_{\C^*}E) \ar[r, hook] & D_\bd	(\Qcoh_{\C^*}E)\ar[r,hook]&D	(\Qcoh_{\C^*}E)\\
 D^b(\coh_{\C^*}E)\ar[r, "\simeq"]  &D_\coh^b	(\Qcoh_{\C^*}E)\ar[r, hook] \ar[u, "U_\coh", "\simeq"' ] & D^b	(\Qcoh_{\C^*}E) \ar[u, "\simeq", "U"']\,.&
 \end{tikzcd}
 $$
\end{proof}

For a locally free coherent sheaf ${\cal F}_0$ we denote by ${\cal F}$ its pull-back to $E$ regarded as $\C^*$-sheaf in the obvious way, and by ${\cal F}\langle k\rangle$ the $\C^*$-sheaf on $E$ obtained from ${\cal F}$ by twisting with the character $z\mapsto z^k$.

\begin{thry} \label{prop-old-l2-2} Let ${\cal T}_0$ be a coherent sheaf on $B$ which generates $D(\Qcoh B)$. Put ${\cal T}\coloneqq  \pi^*({\cal T}_0)$. The family $({\cal T}\langle k\rangle)_{k\in\Z}$ generates $D(\Qcoh_{\C^*} E)$, and classically generates $D^b(\coh_{\C^*} E)$.	
\end{thry}

\begin{proof}  Let ${\cal H}$ be a $\C^*$-equivariant quasi-coherent sheaf on $B$ endowed with the trivial $\C^*$-action. For any open affine subscheme $U\subset B$ we obtain a  $\C[U]$-group scheme $\C^*_U\coloneqq U\times \C^*$ in the sense of \cite[section I.2.1, p. 19]{Ja}. Using  \cite[section 1.2, p. 241]{Tho} we see that the $\C[U]$-module  ${\cal H}(U)$ becomes a  $\C^*_U$-module in the sense of \cite[section I.2.7, p. 19]{Ja}.  By \cite[section I.2.11, p. 30]{Ja} we obtain a decomposition
$${\cal H}(U)=\bigoplus_{\lambda \in\Z} {\cal H}(U)_\lambda 
$$
of ${\cal H}(U)$ as direct sum of $\C[U]$-submodules, each  ${\cal H}(U)_\lambda$ being the submodule ${\cal H}(U)$ on which $\C^*$ acts with weight $\lambda$.  Therefore we have a global direct sum decomposition
\begin{equation}
{\cal H}=\bigoplus_{\lambda \in\Z} {\cal H}_\lambda	
\end{equation}
of ${\cal H}$ as direct sum of quasi-coherent subsheaves. Note that this weight  decomposition holds for an arbitrary quasi-coherent sheaf on $B$ (coherence is not necessary).

Let  now ${\cal F}$ be a $\C^*$-equivariant quasi-coherent sheaf on $E$. The corresponding decomposition 
\begin{equation}\label{decweights}
\pi_*({\cal F})=\bigoplus_{\lambda \in\Z} \pi_*({\cal F})_\lambda	
\end{equation}
combined with \cite[Lemma 28.11.6]{Stack} shows that the functor $\pi_*$ is an equivalence between the category $\Qcoh_{\C^*} E$ and the category of $\Z$-graded quasi-coherent $S^\bullet{\cal E}^\smvee$-modules on $B$. 

Let 
$$\pi^{*0}:\Qcoh B\to \Qcoh_{\C^*} E
$$
be the functor obtained by endowing the pull-back $\pi^*({\cal H})$ of a quasi-coherent sheaf on $B$ with its obvious $\C^*$-structure. Its right adjoint is the functor
$$\pi_{*0}:\Qcoh_{\C^*}E\to  \Qcoh B
$$
given by ${\cal F}\mapsto \pi_*({\cal F})_0=\pi_*({\cal F})^{\C^*}$, so  for any quasi-coherent sheaf ${\cal H}$ on $B$ and $\C^*$-equivariant quasi-coherent sheaf ${\cal F}$ on $E$ we have an identification
$$\Hom_{\Qcoh_{\C^*}E}(\pi^{*0}({\cal H}),{\cal F})=\Hom_{\Qcoh B}({\cal H},\pi_{*0}({\cal F})).
$$
For $k\in\Z$ we get an identification
$$\Hom_{\Qcoh_{\C^*}E}(\pi^{*0}({\cal H})\langle k\rangle,{\cal F})=\Hom_{\Qcoh_{\C^*}E}(\pi^{*0}({\cal H}),{\cal F}\langle -k\rangle)$$
$$=\Hom_{\Qcoh B}({\cal H},\pi_{*0}({\cal F}\langle -k\rangle))=\Hom_{\Qcoh B}({\cal H}, \pi_*({\cal F}\langle -k\rangle)^{\C^*})$$
$$=\Hom_{\Qcoh B}({\cal H}, \pi_*({\cal F})_k), 
$$
which shows that the functor
$$\pi_{*k}:\Qcoh_{\C^*}E\to  \Qcoh B
$$
given by ${\cal F}\mapsto \pi_*({\cal F})_k$ is the right adjoint of the composition 
$$\pi^{*k}\coloneqq \langle k\rangle\circ \pi^{*0}:\Qcoh B\to \Qcoh_{\C^*} E.$$
The functors $\pi_{*k}$, $\pi^{*k}$ are exact, because $\pi^*$, $\pi_*$ are exact. As in the proof of (\ref{adj}) we obtain well-defined, mutually adjoint functors
$$D(\pi_{*k}):D(\Qcoh_{\C^*}E)\to D(\Qcoh B),\ D(\pi^{*k}):D(\Qcoh B)\to D(\Qcoh_{\C^*} E).
$$
Therefore, for any objects ${\cal M}$, ${\cal N}$ of  $D(\Qcoh B)$, $D(\Qcoh_{\C^*} E)$ respectively we have an identification
\begin{equation}\label{adjk}
\Hom_{D(\Qcoh_{\C^*}E)}((D\pi^{*k})({\cal M}),{\cal N})=\Hom_{D(\Qcoh B)}({\cal M},D(\pi_{*k}){\cal N}).	
\end{equation}
 Let ${\cal N}$ be an object of $D(\Qcoh_{\C^*} E)$ such that
 $$\Hom_{D(\Qcoh_{\C^*} E)}({\cal T}\langle k\rangle, {\cal N}[m])=0\ \forall (k,m)\in\Z\times\Z.
 $$
Using (\ref{adjk}) we obtain
$$\Hom_{D(\Qcoh B)}({\cal T}_0, D(\pi_{*k})({\cal N}[m]))=0\ \forall (k,m)\in\Z\times\Z.
$$
Therefore
$$\Hom_{D(\Qcoh B)}\big({\cal T}_0, D(\pi_{*k})({\cal N})[m]\big)=0\ \forall (k,m)\in\Z\times\Z.
$$	
Since ${\cal T}_0$ generates $D(\Qcoh B)$, it follows $D(\pi_{*k})({\cal N})=0$ in $D(\Qcoh B)$ for any $k\in\Z$.  Therefore the complex $D(\pi_{*k})({\cal N})$ is acyclic for any $k\in\Z$. Applying  (\ref{decweights})  to the terms of the complex $D(\pi_*)({\cal N})$ and taking into account that  cohomology commutes with direct sums, it follows that  $D(\pi_*)({\cal N})$ is an acyclic complex. Thus the complex ${\cal N}$ is acyclic, so ${\cal N}=0$ in $D(\Qcoh_{\C^*} E)$.\\

In order to prove that $({\cal T}\langle k\rangle)_{k\in\Z}$ classically generates $D^b(\coh_{\C^*} E)$ note first that ${\cal T}\langle k\rangle$ is a compact object of $D(\Qcoh_{\C^*} E)$ for any $k\in\Z$. Therefore $D(\Qcoh_{\C^*} E)$ is compactly generated by the family $({\cal T}\langle k\rangle)_{k\in\Z}$. By Ravenel-Neeman's Theorem \cite[2.1.2]{BVdB} it follows that $({\cal T}\langle k\rangle)_{k\in\Z}$ classically generates $D(\Qcoh_{\C^*} E)^c$. We claim that $D(\Qcoh_{\C^*} E)^c$ coincides with with the full subcategory %
$${\cal I}\coloneq D_\bdcoh	(\Qcoh_{\C^*}E)$$
 of $D(\Qcoh_{\C^*} E)$ introduced in the proof of Lemma \ref{newLm}. According to this lemma, ${\cal I}$ is formed by the complexes which are isomorphic (in   $D(\Qcoh_{\C^*} E)$) with a bounded complex of coherent $\C^*$-equivariant sheaves. Since any object ${\cal F}$ in $\coh_{\C^*} E$ defines obviously a compact object of $D(\Qcoh_{\C^*} E)$, it follows that $D(\Qcoh_{\C^*} E)^c$ contains the full triangulated subcategory generated by $\coh_{\C^*} E$, so $D(\Qcoh_{\C^*} E)^c\supset{\cal I}$. On the other hand  $D(\Qcoh_{\C^*} E)^c$ is classically generated by the family $({\cal T}\langle k\rangle)_{k\in\Z}$. Since ${\cal I}$ is a thick subcategory of  $D(\Qcoh_{\C^*} E)$  which contains these objects, it follows that ${\cal I}\supset D(\Qcoh_{\C^*} E)^c$. Therefore ${\cal I}=D(\Qcoh_{\C^*} E)^c$ and the last assertion of the theorem follows from Lemma \ref{newLm}.
\end{proof}

\subsection{Graded tilting  on vector bundles}\label{TiltingIH}

We will need the following generalized Beilinson lemma:
\begin{lm}\label{Schwede} Let ${\cal C}$, ${\cal D}$ be a triangulated categories with arbitrary set-indexed  coproducts, $F:{\cal C}\to {\cal D}$ be an exact functor which commutes with set-indexed coproducts, and  let $(T_j)_{j\in J}$ be a family of compact generators of ${\cal C}$	such that 
\begin{enumerate}
\item $(F(T_j))_{j\in J}$ is a family of compact generators of ${\cal D}$.
\item For any pair $(j,n)\in J\times \Z$ the map
$$\Hom_{\cal C}(T_j[n],T_k)\to \Hom_{\cal D}(F(T_j)[n],F(T_k))
$$
induced by $F$ is bijective for any $k\in\Z$.
\end{enumerate}
Then $F$ is an equivalence.
\end{lm}
The case of a single generator  is  \cite[Proposition 3.10]{Sch}, and the proof in the general case follows the same method. We give this proof below for completeness. 
\begin{proof}
Let ${\cal C}'$ be the full subcategory of ${\cal C}$ whose  objects are the objects $Y$ of ${\cal C}$ for which the map
$$\Hom_{\cal C}(T_j[n],Y)\to \Hom_{\cal D}(F(T_j)[n],F(Y)) 
$$
is bijective for any $(j,n)\in J\times \Z$. Since $F$ commutes with coproducts, and $T_j[n]$, $F(T_j)[n]$ are compact objects, it follows that the subcategory ${\cal C}'$ is closed under set-indexed coproducts. Moreover, using the exactness of $F$ and \cite[Lemma 1.1.10]{Nee} it follows that ${\cal C}'$ is closed under extensions, i.e. if two objects of a distinguished triangle are in ${\cal C}'$, then so is the third object of the triangle.  In particular ${\cal C}'$ is a triangulated subcategory. Since ${\cal C}'$ contains the family of generators $(T_j)_{j\in J}$ it follows by \cite[Lemma 2.2.1]{SchSh} that ${\cal C}'={\cal C}$.

Fix now an object $Y$ of ${\cal C}$, and let ${\cal C}_Y$ be the full subcategory of ${\cal C}$ whose  objects are the objects $X$ of ${\cal C}$ for which the map
$$\Hom_{\cal C}(X,Y)\to \Hom_{\cal D}(F(X),F(Y)) 
$$
induced by $F$ is bijective. Since $F$ commutes with direct sums, and the functors $\Hom(-,Y)$, $\Hom(-,F(Y))$ sends coproducts  to products (by the universal property of the coproduct), it follows that ${\cal C}_Y$ is closed under direct sums.  Using the exactness of $F$ and \cite[Remark 1.1.11]{Nee} it follows that ${\cal C}_Y$ is closed under extensions, in particular it is a triangulated subcategory of ${\cal C}$. By the first part of the proof we know that ${\cal C}_Y$ contains the family of generators $(T_j)_{j\in J}$, so ${\cal C}_Y={\cal C}$. This proves that $F$ is fully faithful. 

To prove that $F$ is essentially surjective, let ${\cal D}'$ be the full subcategory of ${\cal D}$ whose objects are the objects of ${\cal D}$ which are isomorphic to an object of the form $F(X)$.  ${\cal D}'$ is closed under the shift functor and coproducts, because $F$ commutes with these operations. We prove that  ${\cal D}'$ is also closed under extensions.  Let
$$U\textmap{\phi} V\to W\to U[1]
$$
be a distinguished triangle with $U$, $V$ objects of ${\cal D}'$. Therefore there exists objects $X$, $Y$ in ${\cal C}$  such that $U\simeq F(X)$, $V\simeq F(Y)$. Fix isomorphisms $u:F(X)\to U$, $v:F(Y)\to V$. We know that $F$ is full, so there exists $f\in \Hom(X,Y)$ such that $F(f)=v^{-1}\phi u$. We can embed $f$ in a distinguished triangle
$$X\textmap{f} Y\to Z\to X[1],
$$
which gives (since $F$ is exact) a distinguished triangle
$$F(X)\textmap{F(f)=v^{-1}\phi u} F(Y)\to F(Z)\to F(X)[1].
$$
In  the following commutative diagram
$$\begin{tikzcd}
F(X)\ar[d, "u"', "\simeq"] \ar[r, "F(f)"] & F(Y)\ar[d, "v", "\simeq"'] \ar[r]&F(Z)\ar[r] \ar[d, dashed, "\simeq"]& F(X)[1]	\\
U\ar[r,"\phi"] & V\ar[r]& W\ar[r]& U[1]
\end{tikzcd}
$$
the rows are distinguished triangles. Therefore $F(Z)$ and $W$ are isomorphic, so $W$ is also an object of ${\cal D'}$. Since ${\cal D}'$ is closed under shifts in both directions, it follows that ${\cal D}'$ is closed under extensions, in particular it is a triangulated subcategory. But we know that ${\cal D}'$ is also closed under  coproducts and contains the family of generators  $(F(T_j))_{j\in J}$. Therefore ${\cal D}'={\cal D}$, which shows that $F$ is essentially surjective.
\end{proof}


In this section 
%
we let again  ${\cal T}_0$ be  a locally free sheaf on $B$ classically generating $D^b(\coh B)$, which satisfies the hypothesis of Corollary \ref{co3} (3), so that ${\cal T}\coloneqq \pi^*({\cal T}_0)$ is a tilting sheaf on $E$.

Using Geometric Tilting Theory and Corollary \ref{co3} (3) we see that the associated graded $\C$-algebra   
$$\Lambda=\End_E({\cal T})=\bigoplus_{m\geq 0} H^0(B, {\cal T}_0^\smvee\otimes {\cal T}_0\otimes S^m{\cal E}^\smvee )
$$
is a  finite $R$-algebra, finitely generated over $\C$, and has finite global dimension. Let $\mathrm{Mod}_\Z\Lambda$ ($\mathrm{mod}_\Z\Lambda$) be the category of (respectively finitely generated) graded right $\Lambda$-modules.  The functor
$$\Hom^{\mathrm{gr}}_{\Qcoh_{\C^*} E}({\cal T},-):\Qcoh_{\C^*} E\to \mathrm{Mod}_\Z\Lambda 
$$
defined by
$${\cal F}\to \bigoplus_{k\in\Z} \Hom_{\Qcoh_{\C^*} E}({\cal T},{\cal F}\langle k\rangle)  
$$
 is left exact. Its right derived functor
 $$R\Hom^{\mathrm{gr}}_{\Qcoh_{\C^*} E}({\cal T},-):D(\Qcoh_{\C^*} E)\to D(\mathrm{Mod}_\Z\Lambda)
 $$
 will be called the graded tilting functor.  For an object ${\cal F}$ in $\Qcoh_{\C^*} E$ we have
 $$H^n(R\Hom^{\mathrm{gr}}_{\Qcoh_{\C^*} E}({\cal T},{\cal F}))=\bigoplus_{k\in \Z} \Ext^n_{\Qcoh_{\C^*} E}({\cal T},{\cal F}\langle k\rangle).
 $$
 In particular, for $m\in\Z$ we have
\begin{equation}\label{F1} 
H^n(R\Hom^{\mathrm{gr}}_{\Qcoh_{\C^*} E}({\cal T},{\cal T}\langle m\rangle))=\bigoplus_{k\in \Z} \Ext^n_{\Qcoh_{\C^*} E}({\cal T},{\cal T}\langle m+k\rangle).
 \end{equation}
 Using \cite[Lemma 2.2.8]{BFK}  we obtain the general formula
 $$\Ext^n_{\Qcoh_{\C^*}E}({\cal F}',{\cal F}'')=\Ext^n_{\Qcoh E}({\cal F}',{\cal F}'')^{\C^*},
 $$
which shows that

$$\Ext^n_{\Qcoh_{\C^*}E}({\cal T}\langle s\rangle, {\cal T}\langle t\rangle)=\Ext^n_{\Qcoh E}({\cal T},{\cal T})_{t-s}.
$$
Since ${\cal T}$ is a tilting sheaf on $E$ it follows that
\begin{equation}\label{F2}\Ext^n_{\Qcoh_{\C^*}E}({\cal T}\langle s\rangle, {\cal T}\langle t\rangle)=
\left\{
\begin{array}{ccc}
\Lambda_{t-s} & \rm for & n=0\\
0 &	 \rm for & n>0
\end{array}\right..
\end{equation}
Combining this formula with (\ref{F1}) we see that $R\Hom^{\mathrm{gr}}_{\Qcoh_{\C^*} E}({\cal T},{\cal T}\langle m\rangle)$ is acyclic in degree $n\ne 0$, and  
\begin{equation}\label{F3}
R\Hom^{\mathrm{gr}}_{\Qcoh_{\C^*} E}({\cal T},{\cal T}\langle m\rangle)=\bigoplus_{k\in\Z} \Lambda_{m+k} =\Lambda\langle m\rangle.	
\end{equation}
Here $\langle m\rangle:\mathrm{Mod}_\Z\Lambda\to \mathrm{Mod}_\Z\Lambda$  is the standard $m$-order shift functor on the category of graded right $\Lambda$-modules.
\begin{lm}\label{shifts}
Let $\Lambda$ be a graded $\C$-algebra of finite global dimension. The family $(\Lambda\langle j\rangle )_{j\in\Z}$ classically generates $D^b(\mathrm{mod}_\Z\Lambda)$ and generates $D(\mathrm{Mod}_\Z\Lambda)$.	
\end{lm}
\begin{proof}
The thick envelope of the family $(\Lambda\langle j\rangle)_{j\in\Z}$ contains all direct sums of the form $\bigoplus_{i\in I} \Lambda\langle k_i\rangle^{\oplus n_i}$
with $I$ finite, $k_i\in\Z$, $n_i\in\N$. Therefore it contains all finitely generated graded projective $\Lambda $-modules. Since $\Lambda$ has finite global dimension, every finitely generated graded $\Lambda$-module has a finite resolution by finitely generated graded projective $\Lambda$-modules. This implies that $(\Lambda\langle j\rangle)_{j\in\Z}$ classically generates $D^b(\mathrm{mod}_\Z\Lambda)$.\\

For the second claim, let $M^\bullet=\bigoplus_{k\in\Z} M^\bullet_k $ be a complex of graded right $\Lambda$-modules such that 
\begin{equation}\label{OrthGr}
\Hom_{D(\mathrm{Mod}_\Z\Lambda)}(\Lambda\langle j\rangle, M^\bullet[n])=0  \ \forall (j,n)\in\Z\times\Z.
\end{equation}
Note that, since $\Lambda(j)$ is a projective object in the category of graded $\Lambda$-modules, we have
$$\Hom_{D(\mathrm{Mod}_\Z\Lambda)}(\Lambda\langle j\rangle, M^\bullet[n])=H^n(\Hom^\bullet(\Lambda\langle j\rangle, M^\bullet))=H^n(M^\bullet\langle -j\rangle).$$ 

Therefore (\ref{OrthGr}) implies that the complex $M^\bullet_k$ is acyclic for any $k\in\Z$, so $M^\bullet=0$ in $D(\mathrm{Mod}_\Z\Lambda)$.
\end{proof}

\begin{thry}\label{gap}
Let $B$ be a smooth projective variety over $\C$, $H$ a finite dimensional complex vector space, and 	 
$$
\begin{tikzcd}[row sep=7ex]
E \ar[r, hook, "i"]  \ar[rr, "\pi", bend left=28, two heads ] &	B\times H^\smvee \ar[r, two heads, "p"] &B
\end{tikzcd}
$$
a sub-bundle of the trivial bundle $B\times H^\smvee$ over $B$. 
%
Let ${\cal T}_0$ be a locally free sheaf on $B$ classically  generating $D^b(\coh B)$, such that ${\cal T}\coloneqq \pi^*({\cal T}_0)$ is a tilting sheaf on $E$. Set $\Lambda\coloneqq \End_E({\cal T})$. Then $\Lambda$ is a graded Noetherian $\C$-algebra of finite graded global dimension, and the graded tilting functor  
$$R\Hom^{\mathrm{gr}}_{\Qcoh_{\C^*} E}({\cal T},-): D(\Qcoh_{\C^*} E)\to D(\mathrm{Mod}_\Z\Lambda)
$$
is an equivalence which restricts to an equivalence
$$D^b(\coh_{\C^*} E)\to D^b(\mathrm{mod}_\Z\Lambda).
$$
\end{thry}
\begin{proof} 	Geometric tilting theory applies, and shows that $\Lambda$ (as ungraded $\C$-algebra) has finite global dimension. By \cite[Theorem II.8.2 p. 122]{NaOy} it follows that $\Lambda$ has also finite graded global dimension. By Theorem \ref{prop-old-l2-2} we know that   $({\cal T}\langle k\rangle)_{k\in\Z}$ is a family of compact generators of the category $D(\Qcoh_{\C^*} E)$. We will show that the graded tilting functor and this family of compact generators satisfy the hypothesis of Lemma \ref{Schwede}.  The  graded tilting functor is exact and commutes with set-indexed coproducts.  Formula (\ref{F3}) shows that 
$$
R\Hom^{\mathrm{gr}}_{\Qcoh_{\C^*} E}({\cal T},{\cal T}\langle m\rangle)=\Lambda\langle m\rangle \ \forall m\in\Z. 
$$
Since $(\Lambda\langle m\rangle)_{m\in\Z}$ is a family of compact generators of $D(\mathrm{Mod}_\Z\Lambda)$ by Lemma \ref{shifts}, we see that the first hypothesis of Lemma \ref{Schwede} is satisfied. To check the second hypothesis, it suffices to note that by (\ref{F2}) one has canonical identifications
$$\Hom_{D(\Qcoh_{\C^*} E)}({\cal T}\langle k\rangle, {\cal T}\langle l\rangle)=\Hom_{\Qcoh_{\C^*} E}({\cal T}\langle k\rangle, {\cal T}\langle l\rangle)=\Lambda_{l-k}=$$
$$=\Hom_{\mathrm{Mod}_\Z\Lambda}(\Lambda\langle k\rangle, \Lambda\langle l\rangle)=\Hom_{D(\mathrm{Mod}_\Z\Lambda)}(\Lambda\langle k\rangle, \Lambda\langle l\rangle).
$$
This shows that Lemma \ref{Schwede} applies, hence 
$$R\Hom^{\mathrm{gr}}_{\Qcoh_{\C^*} E}({\cal T},-): D(\Qcoh_{\C^*} E)\to D(\mathrm{Mod}_\Z\Lambda)
$$
is an equivalence as claimed. To prove that this functor restricts to an equivalence
$$D^b(\coh_{\C^*} E)\to D^b(\mathrm{mod}_\Z\Lambda)
$$ 
it suffices to note that 
\begin{enumerate}[(a)]
\item The thick closure of $(\Lambda\langle m\rangle)_{m\in\Z}$ is $D^b(\mathrm{mod}_\Z\Lambda)$. This was proved in Lemma \ref{shifts}.

\item The canonical functor $D^b(\coh_{\C^*} E)\to 	D(\Qcoh_{\C^*} E)$ defines an equivalence between $D^b(\coh_{\C^*} E)$ and the full subcategory ${\cal I}= D_\bdcoh	(\Qcoh_{\C^*}E)$ of $D(\Qcoh_{\C^*} E)$ of complexes with bounded, coherent cohomology. This is Lemma \ref{newLm}.
\item The thick closure of $({\cal T}\langle k\rangle)_{k\in\Z}$ coincides with ${\cal I}$. This was proved in Theorem \ref{prop-old-l2-2}.
\end{enumerate}
 
\end{proof}
%

%
\def\sv{s^{\hskip-0.2ex\scriptscriptstyle \vee}}

\subsection{A tilting version of the Isik-Shipman theorem }

Consider now an element $s\in H$, defining a regular section of $E^\smvee$ with zero locus $Z(s)\subset B$. The element $s$ defines
%
%
also a degree 1   central non zero-divisor  (denoted by the same symbol) in the graded $\C$-algebra
$$\Lambda=\bigoplus_{m\geq 0} H^0(B, {\cal T}_0^\smvee\otimes {\cal T}_0\otimes S^m{\cal E}^\smvee ).
$$

Let $\langle 1\rangle$ be the shift functor on the category of graded right $\Lambda$-modules, and let
$$\sg: \id_{\mathrm{Mod}_\Z\Lambda}\to \langle 1\rangle
$$
be the natural transformation given by multiplication with $s\in \Lambda_1$.  Denote by $\theta$ the character $\id_{\C^*}$, and by $D(\coh_{\C^*} E, \theta, \sv)$ the derived factorization category associated with the 4-tuple $(\C^*,E, \theta, \sv)$ \cite{Hi1}.
We will use the following result
\begin{thry} \label{FirstEq} With the notations and under the assumptions of Theorem \ref{gap}, let  $s\in H^0(B,{\cal E}^\smvee)$ be a regular section. The tilting sheaf ${\cal T}=\pi^*({\cal T}_0)$ induces an equivalence:
$$\tilde {\cal T}_*: D(\coh_{\C^*} E, \theta, \sv)\to D(\mathrm{mod}_\Z\Lambda,\langle 1\rangle, \sg)
$$	
\end{thry}
\begin{proof}
Since $\Lambda$ is a Noetherian, the method of proof of \cite[Theorem 5.1]{BDFIK} applies. Note however that this proof needs  Theorem \ref{gap}, which is our $\C^*$-equivariant version of Geometric Tilting Theory. \end{proof}
In  the equivalence given by Theorem \ref{FirstEq} one can substitute the derived factorization category $D(\mathrm{mod}_\Z\Lambda,\langle 1\rangle, \sg)$ by  the homotopy category $K(\mathrm{proj}_\Z\Lambda,\langle 1\rangle, \sg)$ of factorisations whose components are finitely generated projective graded $\Lambda$-modules. This is an important progress, because the morphisms in the  category  $K(\mathrm{proj}_\Z\Lambda,\langle 1\rangle, \sg)$ are just homotopy classes of morphisms of factorisations.  The precise statement is
\begin{co} \label{SecondEq}
 With the notations and under the assumptions of Theorem \ref{gap} suppose also that $s\in H^0(B,{\cal E}^\smvee)$ is a regular section. Then the 	sheaf ${\cal T}=\pi^*({\cal T}_0)$ induces an equivalence
$$D(\coh_{\C^*} E,  \theta, \sv)\to K(\mathrm{proj}_\Z\Lambda,\langle 1\rangle, \sg).
$$	

\end{co}

\begin{proof}
As $\Lambda$ has finite graded global dimension,  \cite[Corollary 2.25, p. 210]{BDFIK} applies  and yields an equivalence 
$$K(\mathrm{proj}_\Z\Lambda,\langle 1\rangle, \sg)\simeq D(\mathrm{mod}_\Z\Lambda,\langle 1\rangle, \sg).
$$
\end{proof}

Finally we will identify the homotopy category $K(\mathrm{proj}_\Z\Lambda,\langle 1\rangle, \sg)$ with the triangulated graded singularity category of the quotient algebra $\Lambda/s\Lambda$.
  
\begin{pr} \label{Orlov-newProp} Under the assumptions and with the notations of Theorem \ref{gap} suppose also that $s\in H^0(B,{\cal E}^\smvee)$ is a regular section. Then there exists a natural equivalence 
 \begin{equation}\label{Orlov-new}
\Phi:  K(\mathrm{proj}_\Z\Lambda,\langle 1\rangle, \sg)\textmap{\simeq}  D^{\mathrm{gr}}_{\mathrm{sg}}(\Lambda/s\Lambda).	
 \end{equation}
 \end{pr}
\begin{proof}  
This equivalence is obtained using  a version of Orlov's Theorem \cite[Theorem 3.10]{Or}. Orlov's result gives an equivalence 
$$F:\mathrm{DGrB}(s)\to D^{\mathrm{gr}}_{\mathrm{sg}}(\Lambda/s\Lambda),$$  
where: 
\begin{enumerate}
\item 	$\Lambda=\bigoplus_{i\geq 0}\Lambda_i$ is a {\it connected}, Noetherian algebra of finite global dimension over a field $K$.
\item $s\in\Lambda_n$ is a homogeneous,  central element of positive degree  which is not a zero divisor.
\item $s\Lambda=\Lambda s$ is the two-sided ideal generated by $s$.
\item $\mathrm{DGrB}(s)$ is the triangulated category of graded brains of type $B$ associated with the pair $(\Lambda, s)$ \cite[section 3.1]{Or}. 
\item $D^{\mathrm{gr}}_{\mathrm{sg}}(\Lambda/s\Lambda)$ denotes the graded singularity category of the graded quotient algebra $\Lambda/s\Lambda$.
\end{enumerate}

The construction of the equivalence $\Phi$ starts with the following remark: Orlov's category $\mathrm{DGrB}(s)$ coincides with the full subcategory of $K(\mathrm{proj}_\Z\Lambda,\langle 1\rangle, \sg)$ whose objects are factorizations with (finitely generated) {\it free} graded right $\Lambda$-modules. Moreover, Orlov's functor  $F:\mathrm{DGrB}(s)\to D^{\mathrm{gr}}_{\mathrm{sg}}(\Lambda/s\Lambda)$ defined in \cite[Proposition 3.5]{Or} extends in an obvious way to a functor $\Phi:K(\mathrm{proj}_\Z\Lambda,\langle 1\rangle, \sg)\to   D^{\mathrm{gr}}_{\mathrm{sg}}(\Lambda/s\Lambda)$.

The arguments of \cite[Proposition 3.9]{Or} also apply to the extension $\Phi$, proving that this functor is fully faithful as well. 

The proof is completed by noting that, for an arbitrary (not-necessarily connected) Noetherian algebra $\Lambda$,  the extension $\Phi$ is always essentially surjective.  We indicate briefly the necessary changes to Orlov's proof:
the proof of essential surjectivity in \cite[Proposition 3.10]{Or} is obtained in two steps. First, for an object $T$ in $D_\mathrm{sg}^{\gr}(\Lambda/s\Lambda)$, he obtains a factorization 
\begin{equation}\label{KK}
K^{-1}\textmap{k^{-1}} K^0\textmap{k^0}  K^{-1}(n)
\end{equation}
with $K^0$ free finitely generated and $K^{-1}$ projective finitely generated, which is mapped to $T$ via $\Phi$. Second, using the connectedness of $\Lambda/s\Lambda$, he proves that $K^{-1}$ is  free as well. 

For the essential surjectivity of $\Phi$ the second step is no longer necessary, so it suffices to check that the construction of (\ref{KK}) does not need the connectedness of $\Lambda$. The key ingredients used in this construction    are:
\begin{enumerate}[({I}1)]
\item If $\Lambda$ has finite injective dimension, then $\Lambda/s\Lambda$ has finite injective dimension.

In non-commutative algebra the condition ``$A$ has finite injective dimension" means: ``$A$ has finite injective dimension as both right and left module over itself".  A ring satisfying this condition is called Iwanaga-Gorenstein. \vspace{2mm}
\item The quotient algebra $A\coloneqq \Lambda/s\Lambda$ is a dualizing complex over itself, i.e. the functors
$$D\coloneqq R\Hom_{\mod_\Z A}(-,A):D^b(\mod_\Z A)\to D^b(\mod_\Z A^\op)^\op,$$
$$D\coloneqq  R\Hom_{\mod_\Z A}(-,A):D^b(\mod_\Z A^\op)\to D^b(\mod_\Z A)^\op
$$
are quasi-inverse equivalences.	
\end{enumerate}
(I1) follows using the spectral sequences with second pages
  $$\Ext_A^p(M,\Ext^q_\Lambda(A,\Lambda))\Rightarrow \Ext^n_\Lambda(M,\Lambda),\ \Ext_{A^\op}^p(M,\Ext^q_{\Lambda^\op}(A,\Lambda))\Rightarrow \Ext^n_{\Lambda^\op}(M,\Lambda)
  $$
associated with an $A$-module (respectively $A^\op$-module) $M$   \cite[p. 349]{CaE}, and the short exact sequence
$$0\to \Lambda \textmap{s}\Lambda\to A\to 0
$$
to compute $\Ext^q_\Lambda(A,\Lambda)$, $\Ext^q_{\Lambda^\op}(A,\Lambda)$.
\\ \\ 
(I2) is stated in \cite[Lemma 5.3]{BuSt}. The proof uses only  the assumption  ``$A$ is Iwanaga-Gorenstein".
\end{proof}

Combining   Corollary \ref{SecondEq} and Proposition \ref{Orlov-newProp} with the Isik-Shipman theorem \cite{Is}, \cite{Sh}, one obtains the following result, which gives a tilting description of the derived category of the zero locus of a regular section $s$:

\begin{thry} \label{FourthEq} Let $B$ be a smooth projective variety over $\C$, $H$ a finite dimensional complex vector space, and 	 
$$
\begin{tikzcd}[row sep=7ex]
E \ar[r, hook, "i"]  \ar[rr, "\pi", bend left=28, two heads ] &	B\times H^\smvee \ar[r, two heads, "p"] &B
\end{tikzcd}
$$
a sub-bundle of the trivial bundle $B\times H^\smvee$ over $B$. 
Let ${\cal T}_0$ be a locally free sheaf on $B$  classically  generating $D^b(\coh B)$, such that ${\cal T}\coloneqq \pi^*({\cal T}_0)$ is a tilting sheaf on $E$. Set $\Lambda\coloneqq \End_E({\cal T})$, and let $s\in H$ be an element defining a regular section.   Then one  has an equivalence of triangulated categories:
$$ D^b(\coh Z(s)) \to D^{\mathrm{gr}}_{\mathrm{sg}}(\Lambda/s\Lambda).
$$	
\end{thry}

\subsection{\texorpdfstring{$\C^*$-}{t2}equivariant non-commutative resolutions}\label{Cequiv}

\def\gldim{\mathrm{gldim}}

Let $R$ be a normal, Noetherian domain. We recall (\cite {VdB1}, \cite[sect. 1.1.1]{SVdB}) that a non-commutative  ({\rm nc}) resolution of $R$ is an $R$-algebra of the form $\Lambda=\End_R(M)$, where  $M$ is  a  non-trivial  finitely generated reflexive $R$-module, 	such that $\gldim(\Lambda)<\infty$. A   non-commutative resolution $\Lambda$  is called crepant  if $R$ is a Gorenstein ring, and $\Lambda$ is a  maximal Cohen-Macaulay (MCM) $R$-module.

\begin{dt} 
Let $R$ be a non-negatively graded, normal, Noetherian domain with $R_0=\C$.
A   non-negatively graded  $R$-algebra $\Lambda$ will be called a  graded (crepant)  {\rm nc} resolution  of $R$ if $\Lambda$ is a (crepant)  {\rm nc} resolution in the classical (non-graded) sense. 
\end{dt}

This definition is justified by \cite[Proposition 2.4]{SVdB}, which shows in particular that, denoting by $\mg=\bigoplus_{k>0} R_k\subset R$  the augmentation ideal of $R$,  if
$\Lambda$ is  a  graded (crepant)  {\rm nc} resolution  of $R$, then the $\mg$-completion completion $\hat\Lambda_m$ is a (crepant)  {\rm nc} resolution  of $\hat R_\mg$. Moreover, as we will see in this section, this notion becomes natural in the framework of a Kempf collapsing map (see Lemma \ref{l6} below).

\begin{re} \label{newRe} Suppose that $\Lambda=\End_R(M)$ where $M$ is a non-trivial non-negatively graded finitely generated  $R$-module  which is reflexive  in the non-graded sense, and the inclusion 
$$\bigoplus_{k=0}^\infty \Hom^k(M,M)\subset \End_R(M)
$$
is an equality. Then $\End_R(M)$ endowed with its natural non-negative grading is a graded  {\rm nc} resolution  of $R$, which is crepant if $R$ is Gorenstein ring and $\End_R(M)$ is  an MCM $R$-module.
\end{re}

Let $B$ be a smooth projective variety over $\C$, $H$ a finite dimensional complex vector space, and 
$$
\begin{tikzcd}[row sep=7ex]
E \ar[r, hook, "i"]  \ar[rr, "\pi", bend left=28, two heads ] &	B\times H^\smvee \ar[r, two heads, "p"] &B
\end{tikzcd}
$$
a   vector subbundle of the trivial bundle $B\times H^\smvee$ over $B$. Let $p:B\times H^\smvee\to B$, $q:B\times H^\smvee\to H^\smvee $ be the projections on the two factors,  $\pi=p\circ i$ be the bundle projection of $E$, $\vartheta\coloneqq  q\circ i$, and   $C(E)\coloneqq (q\circ i)(E)$  the  image of the {\it proper} morphism $\vartheta$, endowed with its reduced induced scheme structure. We obtain a surjective projective morphism $\rho:E\to C(E)$ fitting in the commutative diagram:
\begin{equation}\label{DiagramC(E)}
\begin{tikzcd}[row sep=7ex, column sep=7ex]
&[-22pt]E \ar[r, hook, "i"] \ar[dr, "\vartheta"]   \ar[d, "\rho"', two heads] \ar[rr, "\pi", bend left=28, two heads ] &	B\times H^\smvee \ar[d, two heads, "q"]\ar[r, "p"] & B\\ 
\Spec(R)\ar[r, equal]&C(E) \ar[r, hook, "j"] & H^\smvee &
\end{tikzcd}
\end{equation}
\begin{re}\label{AffineConeRem}
Let $E\hookrightarrow B\times H^\smvee$ be  
a   vector subbundle of the trivial bundle $B\times H^\smvee$ over $B$. Then $C(E)$ coincides with the affine cone over the projective variety
$$S(E)\coloneqq  \P(\vartheta)(\P(E))\subset \P(H^\smvee).
$$
\end{re}

The scaling $\C^*$-action on $H^\smvee$ induces  $\C^*$-actions on $E$ and $C(E)$, and $\rho$ is a $\C^*$-equivariant projective morphism.  Since $C(E)$ is $\C^*$-invariant, the associated ideal $I_{C(E)}\subset \C[H^\smvee]$ is homogeneous, so the ring $\C[C(E)]=\C[H^\smvee]/I_{C(E)}$ comes with a natural grading induced by the standard grading of  $\C[H^\smvee]=S^*H$. Put $R\coloneqq \C[C(E)]$. 
%
%
\\

{\it In this section, from here on, we will make the following assumptions: we will suppose that $B$ is connected, $\rho$ is birational, and $C(E)$ is normal.}\\

We will  denote by ${\cal E}$ the locally free sheaf on $B$ associated with $E$.
\begin{lm}\label{l5} One has an isomorphism of graded rings
$$R\simeq \bigoplus_{m\geq 0} H^0(B, S^m {\cal E}^\smvee).
$$	
\end{lm}
\begin{proof}
Since $C(E)=\Spec(R)$ is normal and $\rho$ is proper and birational, we have $\rho_*({\cal O}_E)={\cal O}_{\Spec(R)}$. Therefore
$$R=H^0(\Spec(R),{\cal O}_{\Spec(R)})=H^0(\Spec(R),\rho_*({\cal O}_E))=H^0(E,{\cal O}_E)$$
$$=H^0(B,\pi_*({\cal O}_E))=H^0(B, \bigoplus_{m\geq 0}S^m{\cal E}^\smvee).
$$
The gradings of $R$ and $\bigoplus_{m\geq 0} H^0(B, S^m {\cal E}^\smvee)$ agree via this isomorphism, because $\C^*$ acts with weight $m$ on $H^0(B,S^m{\cal E}^\smvee)$. 	
\end{proof}


\begin{re}\label{NewRemarkNew}
In section \ref{HomGLSMSect}	 we will study a large class of bundles $E$ which fit in a diagram of the form (\ref{DiagramC(E)}) with $H=H^0(B,E^\smvee)$, which satisfy the assumptions of Lemma	 \ref{l5}. These examples will be obtained using homogeneous GLSM presentations.
\end{re}

From now   let  ${\cal T}_0$ be  a locally free sheaf on $B$  classically generating $D^b(\coh B)$, which satisfies the hypothesis of Corollary \ref{co3} (3). Therefore, by this corollary, we know that ${\cal T}\coloneqq \pi^*({\cal T}_0)$ is a tilting sheaf on $E$.
Put $\Lambda\coloneqq \End_E({\cal T})$. Geometric tilting theory (see \cite[sect. 7.6]{HVdB}, \cite[sect. 1.8, 1.9]{BH}, \cite[sect. D]{Le}) implies
\begin{enumerate} 
\item $\Lambda$ is a finite $R$-algebra,  finitely generated as a $\C$-algebra.
\item $\Lambda$ has finite global dimension.
\end{enumerate}
Note also that $\Lambda$ has a natural grading given by the direct sum decomposition:
\begin{equation}\label{Lambda-grading}
\Lambda=H^0(E,\pi^*({\cal T}_0^\smvee\otimes{\cal T}_0))=H^0\big(B, {\cal T}_0^\smvee\otimes{\cal T}_0\otimes (\bigoplus_{m\geq 0}S^m{\cal E}^\smvee)\big).  	
\end{equation}
We are interested in criteria which guarantee that $\Lambda$ is a graded (crepant)  non-commutative resolution of $R$. Let $M$ be the graded $R$-module
$$M\coloneqq  H^0(E,{\cal T})=H^0\big(B,{\cal T}_0\otimes (\bigoplus_{m\geq 0}S^m{\cal E}^\smvee)\big).
$$
\begin{lm}\label{l6} Suppose  that   $\Lambda\coloneqq \End_E({\cal T})$ is a reflexive $R$-module, and ${\cal O}_B$ is a direct summand of ${\cal T}_0$. If the exceptional locus $\mathrm{Exc}(\rho)\subset {C(E)}$ of $\rho$ has codimension $\geq 2$, then $\Lambda$ is   a graded \rm{nc} resolution of $R$.
\end{lm}
\begin{proof} (see \cite[Proposition 3.4]{BLVdB2})  Using the known direct sum decompositions
\begin{align}
R&=H^0(B, \bigoplus_{m\geq 0}S^m{\cal E}^\smvee),\\
M&=H^0\big(B,{\cal T}_0\otimes (\bigoplus_{m\geq 0}S^m{\cal E}^\smvee)\big),\\
\label{LambdaDec}
\Lambda &=H^0\big(B, {\cal T}_0^\smvee\otimes{\cal T}_0\otimes (\bigoplus_{m\geq 0}S^m{\cal E}^\smvee)\big),	
\end{align}
we see that, if ${\cal O}_B$ is a direct summand of ${\cal T}_0$, then $M$ can be identified (as a graded $R$-module) with a direct summand of $\Lambda$. Since we assumed that $\Lambda$ is a reflexive $R$-module, it follows that $M$ is reflexive.
The natural evaluation map
$$\rho_*({\cal E}nd_E({\cal T}))\times \rho_*({\cal T})\to \rho_*({\cal T})
$$
is ${\cal O}_{C(E)}$-bilinear, so it defines a morphism
$$\nu: \rho_*({\cal E}nd_E({\cal T}))\to {\cal E}nd_{C(E)}(\rho_*({\cal T})),
$$
which is obviously an isomorphism  on  ${C(E)}\setminus\mathrm{Exc}(\rho)$.  On the other hand we have
$$H^0({C(E)},\rho_*({\cal E}nd_E({\cal T}))=H^0(E,{\cal E}nd_E({\cal T}))=\Lambda.
$$
Therefore $\Lambda$ is the $R$-module associated with the coherent sheaf $\rho_*({\cal E}nd_E({\cal T}))$ on the affine scheme $\Spec(R)$. 
Since $\Lambda$ is a finitely generated  reflexive $R$-module and $R$ is Noetherian,  it follows that  for any $p\in \Spec(R)$ the stalk $\Lambda_p$ is a reflexive $R_p$-module, so the coherent sheaf $\rho_*({\cal E}nd_E({\cal T}))$ is reflexive. The same argument shows that $\rho_*({\cal T})$ is reflexive, so ${\cal E}nd(\rho_*({\cal T}))$ is reflexive as well. Therefore $\nu$ induces an isomorphism
\begin{equation}\label{iso}
\begin{split}
\Lambda=&  H^0({C(E)},\rho_*({\cal E}nd_E({\cal T}))= H^0({C(E)}\setminus\mathrm{Exc}(\rho),\rho_*({\cal E}nd_E({\cal T}))\textmap{\simeq}\\  
&H^0({C(E)}\setminus\mathrm{Exc}(\rho),{\cal E}nd(\rho_*({\cal T}))) =H^0({C(E)},{\cal E}nd(\rho_*({\cal T})))=\End_{C(E)} (\rho_*({\cal T})).
\end{split} 
\end{equation}

On the other hand $H^0(\rho_*({\cal T}))=H^0(E,{\cal T})=M$, so $M$ is the $R$-module associated with the coherent sheaf $\rho_*({\cal T})$ on the affine scheme ${C(E)}$. Using \cite[Ex. 5.3, p. 124]{Ha} we obtain a natural isomorphism $\End_{C(E)} (\rho_*({\cal T}))=\End_R(M)$, so (\ref{iso}) induces an isomorphism   $\Lambda\textmap{\simeq} \End_R(M)$. Moreover, via this isomorphism one has for any $m\geq 0$ the inclusion
$$H^0 (B, {\cal T}_0^\smvee\otimes{\cal T}_0\otimes S^m{\cal E}^\smvee )\subset \Hom^m_R(M,M).
$$
Therefore, taking into account (\ref{LambdaDec}) we obtain
$$\End_R(M)=\Lambda=\bigoplus_{m\geq 0}H^0 (B, {\cal T}_0^\smvee\otimes{\cal T}_0\otimes S^m{\cal E}^\smvee )\subset \bigoplus_{m\geq 0} \Hom^m_R(M,M),
$$
 so the hypothesis of Remark \ref{newRe} is fulfilled. The claim follows now by this Remark.
 \end{proof}
 
 The following lemma is similar to \cite[Proposition 2.7 (3)]{WZ}.
 
 \begin{lm}\label{l7} Suppose   that  the pre-image $\rho^{-1}(\mathrm{Exc}(\rho))$ of the exceptional locus 	of $\rho$ has codimension $\geq 2$ in $E$. Then $\Lambda$ and $M$ are reflexive $R$-modules, and the natural morphism $\Lambda\to \End_R(M)$ is an isomorphism of  reflexive $R$-modules, and of graded rings. 
 \end{lm}
 \begin{proof}
Since the pre-image $\rho^{-1}(\mathrm{Exc}(\rho))$ has codimension $\geq 2$ in $E$ and $\rho$ is birational, it follows that $\mathrm{Exc}(\rho)$ has also codimension $\geq 2$ in ${C(E)}$, so the functor $\rho_*$ maps reflexive sheaves to reflexive sheaves by \cite[4.2.1]{VdB2}.

 The sheaves ${\cal T}$, ${\cal E}nd_E({\cal T})$ are locally free, hence reflexive.  Therefore   $\rho_*({\cal T})$, $\rho_*({\cal E}nd_E({\cal T}))$ are reflexive coherent  sheaves on ${C(E)}$. The $R$-modules associated with these sheaves are $M$, respectively $\Lambda$. For  $p\in\Spec(R)$ the localizations $M_p$, $\Lambda_p$ are identified with the stalks $\rho_*({\cal T})_p$,  $\rho_*({\cal E}nd_E({\cal T}))_p$ respectively, so they are reflexive $R_p$-modules.  Using \cite[Lemma 15.23.4]{Stack}, it follows that $M$, $\Lambda$ are reflexive $R$-modules. The isomorphism $\Lambda\to \End_R(M)$ is obtained as in the proof of Lemma \ref{l6}. 
\end{proof}

\begin{lm}\label{l8} Suppose $R$ is a Gorenstein ring. $\Lambda$ is an  MCM $R$-module if and only if
$$H^i\big(B,{\cal T}_0^\smvee\otimes{\cal T}_0\otimes S^\bullet{\cal E}^\smvee\otimes \omega_B\otimes\det{\cal E}^\smvee\big)=0\ \forall i>0. 
$$
\end{lm}
\begin{proof}
We follow \cite[Lemma 3.2]{BLVdB2}.   Since ${\cal T}$ is a tilting locally free sheaf on $E$ it follows   that
$$H^i(E,{\cal E}nd({\cal T}))=0 \ \forall i>0. $$
Using the Leray spectral sequence associated with $\rho$, and taking into account that ${C(E)}=\Spec(R)$ is affine, we obtain
$$
H^0({C(E)},R^i\rho_*({\cal E}nd({\cal T})))=0 \ \forall i>0,
$$
so $R^i\rho_*({\cal E}nd({\cal T}))=0$ for any $i>0$. Thus the derived direct image $R\rho_*({\cal E}nd({\cal T}))$ reduces to the direct image $\rho_*({\cal E}nd({\cal T}))$. On the other hand we have
$$\Ext^i_R\big(\End_E({\cal T}),\omega_R\big)=\Ext^i_{C(E)}\big(\rho_* ({\cal E}nd({\cal T})) ,\omega_{C(E)}\big)=$$
$$H^0\big({C(E)},{\cal E}xt^i_{C(E)}\big(\rho_* ({\cal E}nd({\cal T})) ,\omega_{C(E)}\big)\big).$$

Using Weyman's Duality Theorem for proper morphisms \cite[Theorem 1.2.22]{We},  taking into account that ${\cal E}nd({\cal T})$ is locally free and $\rho^{!}\omega_{C(E)}=\omega_E$ \cite[Proposition 1.2.21 (f)]{We}, we obtain an isomorphism
$${\cal E}xt^i_{C(E)}\big(\rho_* ({\cal E}nd({\cal T})) ,\omega_{C(E)}\big)\simeq R^i\rho_* {\cal H}om\big({\cal E}nd({\cal T}),\rho^{!}\omega_{C(E)}\big) $$
$$=R^i\rho_* {\cal H}om\big({\cal E}nd({\cal T}), \omega_E\big) .    
$$
The Leray spectral sequence associated with  $\rho$ gives
$$H^0({C(E)},R^i\rho_* {\cal H}om\big({\cal E}nd({\cal T}), \omega_E\big)=H^i\big(E,{\cal H}om({\cal E}nd({\cal T}),\omega_E)\big),
$$
so we get an isomorphism 
$$\Ext^i_R\big(\End_E({\cal T}),\omega_R)\simeq H^i\big(E,{\cal H}om({\cal E}nd({\cal T}),\omega_E)\big). 
$$
Now use the Leray spectral sequence associated with the affine map $\pi$ and the obvious identification $\omega_E=\pi^*(\omega_B\otimes\det({\cal E}^\smvee))$. We  obtain an isomorphism
$$H^i\big(E,{\cal H}om({\cal E}nd({\cal T}),\omega_E)\big)=H^i\big(B,{\cal T}_0^\smvee\otimes{\cal T}_0\otimes(\bigoplus_{m\geq 0}S^m{\cal E}^\smvee)\otimes \omega_B\otimes\det{\cal E}^\smvee\big),
$$ 
which completes the proof.
\end{proof}
\begin{co}\label{co9} Suppose  that  
  $R$ is a Gorenstein ring,  the exceptional locus $$\mathrm{Exc}(\rho)\subset {C(E)}$$ has codimension $\geq 2$,   ${\cal O}_B$ is a direct summand of ${\cal T}_0$, and $\det({\cal E})\simeq \omega_B$. 	Then $\Lambda$ is a graded crepant {\rm nc} resolution of $R$.
\end{co}
\begin{proof}
 By Corollary \ref{co3} (3) and Lemma  \ref{l8}, $\Lambda$ is an  MCM $R$-module. Taking into account that $R$ is Gorenstein,  it follows by \cite[Lemma 4.2.2 (iii)]{Bu} that $\Lambda$ is reflexive. Since $M$ is a direct summand of $\Lambda$ (see the proof of Lemma \ref{l6}), $M$ is reflexive as well.
	
\end{proof}

Using Lemmas \ref{l7}, \ref{l8}, and Corollary \ref{co9} we obtain
\begin{thry}\label{co10} With the notations and under the assumptions above we have:
\begin{enumerate}
\item  Suppose that
\begin{enumerate}[(i)]
\item $\mathrm{cod}_E\big(\rho^{-1}(\mathrm{Exc}(\rho)))\geq 2$, or  
\item $\mathrm{cod}_{C(E)}\big(\mathrm{Exc}(\rho)\big)\geq 2$, ${\cal O}_B\subset {\cal T}_0$ is a direct summand, and $\Lambda$ is reflexive.
 \end{enumerate}
 Then $\Lambda$ is isomorphic to $\End_R(M)$ and is a graded {\rm nc} resolution of $R$.
\item Suppose $R$ is a Gorenstein ring, $\det({\cal E})\simeq \omega_B$, and that
\begin{enumerate}[(i)]
\item 	$\mathrm{cod}_E\big(\rho^{-1}(\mathrm{Exc}(\rho)))\geq 2$, or 
\item $\mathrm{cod}_{C(E)}\big(\mathrm{Exc}(\rho)\big)\geq 2$, and ${\cal O}_B\subset {\cal T}_0$ is a direct summand.
	\end{enumerate}
Then $\Lambda$ is isomorphic to $\End_R(M)$ and is a graded crepant {\rm nc} resolution of $R$.
\end{enumerate}

\end{thry}

All our examples will be obtained using homogeneous GLSM presentations as explained in Remark 
\ref{NewRemarkNew}. For this class of examples  we will prove an explicit Gorenstein criterion for the ring $R$.

\section{Homogeneous gauged linear sigma models}\label{GeomSect}

\subsection{GIT for  representations of reductive groups}

Let $U$ be a finite dimensional complex vector space.  For a morphism $\psi:\C^*\to \GL(U)$ and an integer $m\in\Z$ we denote by $U^{\psi}_m\subset U$ the linear subspace corresponding to the weight $\zeta\to \zeta^m$, and by $\Wg(\psi)\subset\Z$ the set of weights:
$$\Wg(\psi)=\{m\in \Z|\ U^{\psi}_m\ne\{0\}\}.
$$
The weight decomposition of $U$ reads
$$U=\bigoplus_{m\in\Z} U^{\psi}_m=\bigoplus_{m\in\Wg(\psi)} U^{\psi}_m.
$$
For a vector $u\in U$ we denote by $u^\psi_m$ its $U^{\psi}_m$-component.

Let $\alpha:G\to \GL(U)$ be a linear representation of a complex reductive group $G$ on  $U$, and let $\chi:G\to\C^*$ be a character of $G$.  A vector $u\in U$ is called 
\begin{itemize}
\item $\chi$-semistable, if there exists $n>0$ and  
$$f\in \C[U]^G_{\chi^n}\coloneqq \{f\in \C[U]| f(gu)=\chi(g)^n f(u)\ \forall g\in G\ \forall u\in U\}$$
 such that $f(u)\ne 0$.
\item $\chi$-stable	if it is $\chi$-semistable, the stabilizer $G_u$ is finite, and the orbit $Gu$ is closed in the Zariski open subset of $\chi$-semistable vectors.
\item $\chi$-unstable, if it is not $\chi$-semistable.
\end{itemize}

We will denote by $U_{\st}^\chi$, $U_{\ss}^\chi$, $U_{\us}^\chi$ the sets of  $\chi$-stable, $\chi$-semistable and $\chi$-unstable  points of $U$.
The $\chi$-(semi)stability condition  coincides  with the classical GIT (semi)stability condition associated with the linearization of $\alpha$ in the $G$-line bundle ${\cal O}_U\otimes{\chi^{-1}}$ \cite[Lemma 2.4]{Ho}.

 For a  morphism of algebraic groups $\xi:\C^*\to G$ let  $\langle \chi,\xi\rangle\in\Z$ be the degree of the composition $\chi\circ\xi:\C^*\to\C^*$. The limit $\lim_{\zeta\to 0} (\alpha\circ\xi)(\zeta)(u)$ exists if and only if $u\in \bigoplus_{m\geq 0} U^{\alpha\circ\xi}_m$. Define
$$\mu^\chi(u,\xi)\coloneqq \left\{
\begin{array}{ccc}
\infty 	&\rm if & \exists m\in\Z_{<0} \hbox{ such that } u^{\alpha\circ\xi}_m\ne 0\\
\langle \chi,\xi\rangle & \rm if & u\in \bigoplus_{m\in\Z_{\geq 0}} U^{\alpha\circ\xi}_m.
\end{array}
\right.
$$
Denote by $e$ the unit element of $G$ and also the trivial morphism $\C^*\to G$. Using the Hilbert-Mumford stability criterion for linear actions (\cite{He}, \cite{Ho}, \cite{Te}) we obtain
\begin{pr}\label{HM}
A vector $u\in U$ is $\chi$-stable ($\chi$-semistable) if and only if for any   $\xi\in \Hom(\C^*,G)\setminus\{e\}$ for which $\langle \chi,\xi\rangle\leq 0$ (respectively $\langle\chi,\xi\rangle < 0$) there exists $m\in \Z_{<0}$ such that  $u^{\alpha\circ\xi}_m\ne 0$.
\end{pr}

Equivalently,
\begin{re} A vector $u\in U$ is $\chi$-unstable if and only if there exists a   morphism $\xi:\C^*\to G$ such that $\langle \chi,\xi\rangle  < 0$ and 
$u\in \bigoplus_{m\in \Z_{\geq 0}}U^{\alpha\circ\xi}_m$.
\end{re}

Let  $\Dg:\Hom(\C^*,G)\to \g$ be the injective map defined by the condition
%
 $\Dg(\xi)=d_1\xi(1)$, where  $d_1\xi:\mathrm{Lie}(\C^*)=\C\to\g$ stands for the differential of $\xi$ at 1. Note that for any torus $T\subset G$ the image $\Dg(\Hom(\C^*,T))\subset \g$ is a free $\Z$-submodule of rank   $\dim(T)$.

 Let  $\llangle\cdot,\cdot\rrangle:\g\times\g\to\C$ be  an $\ad_G$-invariant symmetric, bilinear form on the Lie algebra $\g$ of $G$ with the property that for every torus $T\subset G$, its restriction 
 $$
 \Dg(\Hom(\C^*,T))\times \Dg(\Hom(\C^*,T))\to\C
 $$
 is an inner product with rational coefficients (see \cite[section 2.1]{He}). This condition implies 
 $$\llangle \Dg(\xi) ,\Dg(\xi)\rrangle\in [0,\infty)\cap \Q \ \ \forall \xi\in\Hom(\C^*,G).$$
We obtain an $\ad_G$-invariant norm on $\Hom(\C^*,G)$ given by
 $$\vertiii{\xi}\coloneqq \sqrt{\llangle \Dg(\xi) ,\Dg(\xi)\rrangle}.
 $$
 
  Let $u\in U_{\us}^\chi$. An indivisible morphism $\xi\in \Hom(\C^*,G)\setminus \{e\}$ is called an optimal destabilizing morphism for $u$ if it realizes the negative minimum of the map $\alpha\mapsto \frac{1}{\vertiii{\alpha}}\mu^\chi(u,\alpha)$, i.e. if
  $$\frac{\langle \chi,\xi\rangle }{\vertiii{\xi}}=\inf \bigg\{\frac{\langle \chi,\beta\rangle }{\vertiii{\beta}}\ \vline\  \beta\in \Hom(\C^*,G)\setminus \{e\},\ u\in \bigoplus_{m\in \Z_{\geq 0}}U^\beta_m \bigg\}.
  $$
  
For a vector $u\in U^\chi_\ss$ let $\Xi^\chi(u)$  be the set of optimal destabilizing morphisms  for $u$. Recall  that any morphism $\xi\in \Hom(\C^*,G)$ defines a parabolic subgroup $P(\xi)\subset G$ given by
$$P(\xi)\coloneqq \{g\in G|\ \lim_{\zeta\to 0} \xi(\zeta) g \xi(\zeta)^{-1}\hbox{ exists}\}.
$$ 
By a result of Kempf \cite{Ke1} the parabolic subgroup $P^\chi(u)$ associated with an element $\xi\in \Xi^\chi(u)$ is independent of $\xi$. Moreover, the set $\Xi^\chi(u)$ is an orbit with respect to the action of  $P^\chi(u)$ by conjugation on  $\Hom(\C^*,G)\setminus \{e\}$.  
\begin{ex}\label{ExGrass}
Let $V$, $Z$ be non-trivial, finite dimensional complex vector spaces, and let $\alpha$ be the standard representation of $\GL(Z)$ on $\Hom(V,Z)$.  The set of characters of $\GL(Z)$ is $\{\det^t|\ t\in\Z\}$. Denoting by $\Hom(V,Z)^{\rm epi}\subset \Hom(V,Z)$ the open subspace of epimorphisms, one has
$$\Hom(V,Z)_{\ss}^{\det^t}=\left\{
\begin{array}{ccc}
\Hom(V,Z)^{\rm epi} &\rm if & t>0\\	
\Hom(V,Z) &\rm if & t=0\\
\emptyset  &\rm if & t<0
\end{array}
\right..
$$	
The bilinear map
$$\gl(Z)\times\gl(Z)\ni (x,y)\mapsto \llangle x,y\rrangle\coloneqq \tr(xy)
$$
is $\ad_{\GL(Z)}$-invariant, and satisfies the rationality condition mentioned above.
\begin{itemize}
\item Let $t>0$. A morphism $\xi\in \Hom(\C^*,G)\setminus \{e\}$ is   optimal destabilizing  for 
$$u\in \Hom(V,Z)\setminus \Hom(V,Z)^{\rm epi}=\Hom(V,Z)_{\us}^{\det^t}$$
 if and only if there exists a  complement $\Gamma$ of $I\coloneqq \im(u)$  in $Z$ such that, with respect to the direct sum decomposition $Z=I\oplus \Gamma$, one has
\begin{equation}\label{optimal+GR}\xi(\zeta)=\begin{pmatrix}
\id_I & 0\\
0 & \zeta^{-1}\id_\Gamma	
\end{pmatrix} \ \forall \zeta\in\C^*.
\end{equation}
In this case one has $\Hom(V,Z)_{\ss}^{\det^t}=\Hom(V,Z)_{\st}^{\det^t}$, and  the group $\GL(Z)$ acts freely on this space. Putting $N\coloneqq \dim(V)$, $k\coloneqq \dim(Z)$ the corresponding GIT quotient is
$$\qmod{\Hom(V,Z)^{\rm epi}}{\GL(Z)}=\Gr_{N-k}(V)=\Gr_{k}(V^\smvee).
$$

\item Let $t<0$. In this case any vector $u\in \Hom(V,Z)$ is $\det^t$-unstable, and has a unique  optimal destabilizing morphisms which is $\zeta\mapsto\zeta\id_Z$.
\end{itemize}
\end{ex}

\begin{pr} \label{new} Let $\alpha:G\to \GL(U)$, $\beta:G\to \GL(F)$ be linear representations of $G$, and $\chi\in \Hom(G,\C^*)$ a character. Then
\begin{enumerate}
\item $U_{\ss}^\chi\times F\subset (U\times F)_{\ss}^\chi$,
\item Suppose that	for every $u\in U_\us^\chi$ and $\xi\in\Xi^\chi(u)$ one has $\Wg(\beta\circ \xi)\subset\Z_{\geq 0}$. Then
 $$U_{\ss}^\chi\times F= (U\times F)_{\ss}^\chi\ .
 $$
\end{enumerate}
\end{pr}
\def\Spec{\mathrm{Spec}}
\begin{proof}
 Let $\alpha\beta:G\to \GL(U\times F)$ be the morphism induced by the pair $(\alpha,\beta)$.  \\  
 (1) The first claim follows from Proposition \ref{HM}	 using the equality
 $$(U\times F)^{\alpha\beta\circ\xi}_m=U^{\alpha\circ\xi}_m\times F^{\beta\circ\xi}_m\ \forall m\in\Z.$$
(2) Taking into account (1) it suffices to prove the inclusion $(U\times F)_{\ss}^\chi\subset U_{\ss}^\chi\times F$ or, equivalently,  $ U_{\us}^\chi\times F\subset  (U\times F)_{\us}^\chi$.  Let $(u,y)\in  U^\chi_\us\times F$. We claim that any optimal destabilizing morphism $\xi\in \Xi^\chi(u)$ destabilizes the pair $(u,y)$. Indeed, since $\xi$ destabilizes $u$, it follows that $\langle \chi,\xi\rangle <0$, and  $u\in \bigoplus_{m\in \Z_{\geq 0}}U^{\alpha\circ\xi}_m$. By assumption one has $F=\bigoplus_{m\in \Z_{\geq 0}} F^{\beta\circ \xi}_m$, so $(u,y)\in \bigoplus_{m\in \Z_{\geq 0}}(U\times F)^{\alpha\beta\circ\xi}_m$.

\end{proof}

\subsection{GLSM presentations}

Let   $B$ be a complex projective manifold, $\pi:E\to B$    a  rank $r$ vector bundle on $B$,  and $s\in\Gamma(E^\smvee)$  a  section in the dual bundle. 

\begin{dt}\label{GLSMDef}
An algebraic geometric GLSM presentation of the pair $(E\stackrel{\pi}{\to}B,s)$ is the data of a 4-tuple $(G,U,F,\chi)$, where $G$ is a complex reductive group,  $U$ and $F$ are finite dimensional $G$-representation spaces, and $\chi\in \Hom(G,\C^*)$ is a character  such that the following conditions are satisfied:
\begin{enumerate} 
\item $U_\st^\chi=U_{\ss}^\chi$  and  $G$ acts freely	 on this set.
\item The base $B$ coincides with the quotient $U_{\ss}^\chi/G$, and the vector bundle $E$ is the   $F$-bundle associated with the principal $G$-bundle $p:U_{\ss}^\chi\to  B$ and the representation  space $F$.
\item \label{signs} Any optimal destabilizing morphism of any unstable point $u\in U^\chi_\us$ acts with non-negative weights on $F$.
\item \label{extension-eq} The $G$-equivariant map $\hat s:U_{\ss}^\chi\to F^\smvee$ corresponding to $s$ extends to a $G$-equivariant  polynomial map  $\sigma:U\to F^\smvee$.

%
\end{enumerate}
\end{dt}

Note that the map $\sigma:U\to F^\smvee$ is a covariant extension of the $G$-equivariant map $\hat s:U_{\ss}^\chi\to F^\smvee$ corresponding to $s$.

Many interesting  manifolds (e.g. Grassmann manifolds, Flag manifolds, projective toric manifolds) can be obtained as quotients of the form $B=U^\chi_\ss/G$ for a pair $(U,\chi)$ satisfying condition (1) in Definition \ref{GLSMDef}. For  such a quotient manifold $B$ one obtains submanifolds $X\subset B$ defined as zero loci of regular sections $s$ in   associated bundles of the form $E^\smvee= U^\chi_\ss\times_G F^\smvee$. This very general construction method yields a large class of algebraic manifolds with interesting properties (see section \ref{ExSect}).

To explain the role of the fourth condition in this definition note first that, giving a $G$-covariant $\sigma:U\to F^\smvee$ is equivalent to defining a  $G$-invariant,  polynomial map  $\sigma^\smvee: U\times F\to \C$, which is linear with respect to the second argument. The map $\sigma^\smvee$ is given by
$$\sigma^\smvee (u,z)=\langle \sigma(u),z\rangle.
$$

The   bundle $E$ is  the associated bundle $U_\ss^\chi\times_G F$. Denote by $q:U_\ss^\chi\times  F\to E$ the projection map. Then $\resto{\sigma^\smvee}{U_\ss^\chi\times F}=s^\smvee\circ q$, so   $\sigma^\smvee$ is an extension of the pull back   of the potential $s^\smvee:E\to\C$ intervening in the gauged LG model associated with $s$ (see section \ref{LGmodels}). 

 In many cases the complement of $U_\st^\chi$ in $U$ has codimension $\geq 2$, hence the existence of a regular $G$-equivariant extension $\sigma:U\to F^\smvee$ of $\hat s$ follows automatically.   Therefore in this way we obtain a large class of triples satisfying   Definition \ref{GLSMDef}. 
%

Condition (\ref{signs}) in Definition \ref{GLSMDef} implies that, for every $u\in U_\us^\chi$ and any $\xi\in \Xi^\chi(u)$, the $\C^*$-representation $\beta^\smvee\circ\xi$ on $F$ has only non-negative weights. Therefore, by Proposition \ref{new} 
\begin{pr} \label{SSProd} 
Let $(G,U,F,\chi)$ be an algebraic geometric GLSM presentation of $(E\textmap{\pi} B,s)$ with $s\in \Gamma(E^\smvee)$. Then 
\begin{equation}\label{first-char}
U_{\ss}^\chi\times F=(U\times F)_{\ss}^\chi,
\end{equation}
so the   bundle $E=U_{\ss}^\chi\times_G F$ coincides with the GIT quotient $(U\times F)_{\ss}^\chi/G$. 
\end{pr}

We will see that this proposition has important consequences (see Proposition \ref{KempfProp} in the next section).

 \subsection{Homogeneous GLSM presentations}
 \label{HomGLSMSect}
 
 In this section we introduce a homogeneity condition for algebraic geometric GLSM presentations which has interesting geometric consequences.
 
 Let   $B$ be a complex projective manifold, $\pi:E\to B$   a  rank $r$ vector bundle on $B$,  and $s\in\Gamma(E^\smvee)$   a  section in the dual bundle.
\begin{dt}\label{HomGLSMDef} 
An algebraic geometric GLSM presentation $(G,U,F,\chi)$
of  the pair $(E\textmap{\pi} B,s)$   will be called homogeneous, if there exists a right action $\gamma:U\times{\cal G}\to U$ by linear automorphisms  of a connected, complex reductive group ${\cal G}$ on  $U$ such that:
 \begin{enumerate}
 \item   $\gamma$ commutes with the fixed $G$-action on $U$. 
 \item 	 The  induced action $\gamma_B:B\times{\cal G}\to B$ on the quotient $B=U^\chi_\ss/G$ is transitive.
 \item Let ${\cal P}_0\subset {\cal G}$ be the stabilizer of a fixed point $b_0=Gu_0\in B$. The group morphism  $\lambda_0:{\cal P}_0\to G$   defined by 
 $$\lambda_0(y)u_0=    u_0 y,\ \forall y\in {\cal P}_0$$
 is surjective. 
\end{enumerate}
\end{dt} 
Note that the third condition is independent of the pair $(b_0,u_0)$. Since the quotient ${\cal P}_0\backslash{\cal G}\simeq B$ is projective, it follows that ${\cal P}_0$ is a parabolic subgroup of ${\cal G}$. On the other hand, the first condition implies that $U^\chi_\ss$ is ${\cal G}$-invariant, and the induced ${\cal G}$-action on $U^\chi_\ss$ induces a  fibrewise linear  action (which lifts $\gamma_B$) on any vector bundle on $B$ which is associated with the principal $G$-bundle $U^\chi_\ss\to B$. In other words, any such associated vector bundle is naturally a homogeneous vector bundle on the ${\cal G}$-manifold $B$. In particular $E$, $E^\smvee$ become  homogeneous vector bundles on $B$, and $H^0(E^\smvee)^\smvee$ is naturally  a representation space of ${\cal G}$.
\vspace{2mm}

Let $B$ be a projective variety, and  $\pi:E\to B$  be a vector bundle on $B$ such   that $E^\smvee$ is globally generated. This implies that the evaluation map $\vartheta:E\to H^0(E^\smvee)^\smvee$ is fibrewise injective, so it identifies  $E$ with a subbundle of the trivial bundle $B\times H^0(E^\smvee)^\smvee$.  Therefore, putting $C(E)\coloneqq \im(\vartheta)$ we obtain a commutative diagram
\begin{equation}\label{DiagEKempf}
\begin{tikzcd}
E  \ar[d, two heads, "\rho" '] \ar[r, hook] \ar[dr, "\vartheta"] & B\times H^0(E^\smvee)^\smvee  \ar[d, two heads]\\
C(E) \ar[r, hook] & H^0(E^\smvee)^\smvee	
\end{tikzcd}
\end{equation}
where $\rho$ is induced by $\vartheta$.  Recall (see Remark \ref{AffineConeRem}) that $C(E)$ coincides with the affine cone over the projective variety
$$S(E)\coloneqq \P(\vartheta)(\P(E))\subset \P(H^0(E^\smvee)^\smvee).
$$
\begin{pr}\label{KempfProp}
Let $(G,U,F,\chi)$ be a homogeneous, algebraic geometric GLSM presentation  of $(E\textmap{\pi} B,s)$. Suppose that $E^\smvee$ is globally generated.  
Then
\begin{enumerate}
\item The cone $C(E)\subset H^0(E^\smvee)^\smvee$ is a Cohen-Macaulay normal variety.
\item Suppose that $\dim(C(E))=\dim(E)$. Then the morphism $\rho: E \to C(E)$ induced by the proper morphism $\vartheta: E\to H^0(E^\smvee)^\smvee$ is birational, and $C(E)$ has rational singularities.
\item Suppose that $\dim(C(E))=\dim(E)$ and $\mathrm{codim}(U_\us^\chi)\geq 2$. Then there exists an isomorphism $\eta:U\times F\catqot  G\textmap{\simeq} C(E)$ such that the diagram
$$
\begin{tikzcd}
U\times F\catqot_\chi \ar[r, "\simeq"] \ar[d, "q^\chi"']G&E\ar[d, "\rho"] \\
U\times F\catqot G \ar[r, "\simeq", "\eta"']&C(E)	
\end{tikzcd} 
$$
is commutative, in particular the cone $C(E)$ can be identified with the affine  GIT quotient $\Spec(\C[U\times F]^G)$.
\end{enumerate}

\end{pr}
\begin{proof}
(1) The linear subspace $\vartheta(E_{b_0})\subset H^0(E^\smvee)^\smvee$ is ${\cal P}_0$-invariant.  Moreover, the left ${\cal P}_0$-action on $\vartheta(E_{b_0})$ is induced is induced by the $G$-action on $\{u_0\}\times F\simeq F$ via the group morphism $\lambda_0$ intervening  in Definition \ref{HomGLSMDef}. Since $\lambda_0$ is surjective, and $G$ is reductive, it follows that  the ${\cal P}_0$ representation space $\vartheta(E_{b_0})$ is completely reducible. Since ${\cal G}$ acts transitively on $B$ one has $E={\cal G}E_{b_0}$, so
$$C(E)\coloneqq \vartheta(E)=\vartheta({\cal G}E_{b_0})={\cal G}\vartheta(E_{b_0}).
$$
The first claim follows now from \cite[Theorem 0]{Ke}.
\vspace{2mm}\\ 
(2) Since $\dim(C(E))=\dim(E)$ it follows by \cite[Proposition 2(c)]{Ke} that $\rho$ is birational, so the claim follows from the second statement of  \cite[Theorem 0]{Ke}.
\vspace{2mm}\\  
(3) Since $\mathrm{codim}(U\setminus U^\chi_\ss)\geq 2$, we also have $\mathrm{codim}(U\times F)\setminus (U\times F)^\chi_\ss\geq 2$ by Proposition \ref{SSProd}, so the composition 
$$(U\times F)^\chi_\ss\to  E\textmap{\vartheta} H^0(E^\smvee)^\smvee $$
 extends to a $G$-invariant morphism  $U\times F\to H^0(E^\smvee)^\smvee$. Therefore we obtain a $G$-invariant, surjective extension
$$\tilde \Sigma: U\times F\to C(E)
$$
of the composition $\Sigma:(U\times F)^\chi_\ss\to  E\textmap{\rho} C(E)$.  We claim that the   morphism 
$$\eta: U\times F\catqot G\to C(E)$$ 
induced by $\tilde \Sigma$  is an isomorphism. Since $\eta$ is a morphism of affine schemes, it suffices to prove that the induced ring morphism 
$$\eta^*:\C[C(E)]=H^0({\cal O}_{C(E)})\to \C[U\times F]^G$$ 
is an isomorphism. The inclusion morphism $j: (U\times F)^\chi_\ss\hookrightarrow U\times F$ induces a restriction monomorphism 
$$ j^*:\C[U\times F]^G\to \C[(U\times F)^\chi_\ss]^G=H^0({\cal O}_{E}).$$
The composition $j^*\circ \eta^*: H^0({\cal O}_{C(E)})\to H^0({\cal O}_{E})$ coincides with the morphism $\rho^*$ induced by $\rho$. On the other hand, since $\rho$ is proper and birational,  {\it and $C(E)$ is normal}, it follows that $\rho_*({\cal O}_{E})={\cal O}_{C(E)}$. Therefore $\rho^*=j^*\circ \eta^*$ is an isomorphism. Since $j^*$ is a monomorphism, it follows that $j^*$ and $\eta^*$ are both isomorphisms.
\end{proof}

Let $G$ be a connected reductive group, $W$   a finite dimensional $G$-representation space. The set of stable points with respect to the trivial character of $G$ is
$$W_\st=\{w\in W|\ Gw \hbox{ is closed},\ G_w\hbox { is finite}\}.
$$
This set is Zariski open in $W$.

\begin{lm}\label{GorCritLm} (Gorenstein criterion) Suppose that $\mathrm{codim}(W\setminus W_\st)\geq 2$, and the $G$-representation $\det W$ is trivial. Then $\C[W]^G$ is a Gorenstein ring.	 
\end{lm}
\begin{proof}
The closed set 
$$W_s\coloneqq \{w\in W|\ \dim(G_w)>0\}
$$
defined in \cite[p. 40]{Kn2} is contained in  $W\setminus W_\st$, so the assumption $\mathrm{codim}(W\setminus W_\st)\geq 2$  implies  $\mathrm{codim}(W_s)\geq 2$. The same condition also implies that $W_\st\ne\emptyset$, in particular a general $G$-orbit in $W$ is closed. Since $G$ is connected, the determinant $\det(\ad)$ of the adjoint representation is  trivial (because any connected reductive group is unimodular), so  $\det(W)=\det(\ad)$. The claim follows from Knop's criterion \cite[Satz 2, p. 41]{Kn2}. 
\end{proof}
See \cite[5.1]{SVdB} for further remarks. Using Proposition \ref{KempfProp}, a well-known Cohen-Macaulay criterion for affine GIT quotients \cite[Corollaire, p. 66]{Bout} and Lemma \ref{GorCritLm}, we obtain:

\begin{co} \label{PropertiesC(E)}
 Let $(G,U,F,\chi)$ be a homogeneous, algebraic geometric GLSM presentation  of $(E\textmap{\pi} B,s)$. Suppose that $E^\smvee$ is globally generated, $\dim(E)=\dim(C(E))$ and $\mathrm{codim}(U_\us^\chi)\geq 2$. Then
\begin{enumerate} 
\item The cone $C(E)$ is normal, Cohen-Macaulay,  has rational singularities, and is canonically isomorphic to the affine quotient $\Spec(\C[U\times F]^G)$. 
\item Suppose that $G$ is connected, the $G$-representation $\det(U\times F)$ is trivial, and  $\mathrm{codim}\big((U\times F)\setminus(U\times F)_\st\big)\geq 2$. Then $C(E)$ is also Gorenstein.
\end{enumerate}
\end{co}

Taking into account Remark \ref{AffineConeRem} we obtain 
    
\begin{co}\label{PropertiesS(E)}
 Let $(G,U,F,\chi)$ be a homogeneous, algebraic geometric GLSM presentation  of $(E\textmap{\pi} B,s)$. Suppose that $E^\smvee$ is globally generated, $\dim(E)=\dim(C(E))$ and $\mathrm{codim}(U_\us^\chi)\geq 2$. Then the projective variety $S(E)\coloneqq \im(\P(\vartheta))\subset \P(H^0(E^\smvee)^\smvee)$ 	has the following properties:
 \begin{enumerate}
 \item $S(E)$ is projectively normal, arithmetically	Cohen-Macaulay and   its affine cone has rational singularities.
 \item If $G$ is connected, the 1-dimensional $G$-representation $\det(U\times F)$ is trivial, and   $\mathrm{codim}\big((U\times F)\setminus(U\times F)_\st\big)\geq 2$, then $S(E)$ is also arithmetically Gorenstein.
 \end{enumerate}

\end{co}

\subsection{Geometric applications}\label{ExSect}

\subsubsection{A general set up}

In this section we identify an important class of examples to which Theorem \ref{FourthEq} can be applied. Let $V$, $Z$ be complex vector spaces of dimensions $N$, $k$ respectively with $1\leq k< N$. Choose $r\in \N_{>0}$, $(d_1,\dots,d_r)\in \N_{>0}^r$, and for any $1\leq i\leq r$,  choose a $\GL(Z)$-invariant subspace $F_i$ of the tensor power $\otimes^{d_i}Z^\smvee$.  Thus $F_i^\smvee$ is a polynomial $\GL(Z)$-representation of degree $d_i$ \cite[5.3, 5.8, 5.9]{KrPr}. Put 
$$F\coloneqq \bigoplus_{i=1}^r F_i.$$

The class of these $\GL(Z)$-representations coincides with the class of duals of finite dimensional polynomial representations of $\GL(Z)$, i.e. each such $F^\smvee$ is isomorphic to a finite direct sum of Schur modules
$$F^\smvee=\bigoplus_{\lambda\in P(k)} N_\lambda\otimes S^\lambda Z.
$$
Here $P(k)$ denotes the set of partitions $\lambda=(\lambda_1,\cdots,\lambda_k)$ with $\lambda_1\geq\cdots\geq\lambda_k\geq 0$.

Now consider the $\GL(Z)$ representation $U=\Hom(V,Z)$. Choosing $t\in \N_{>0}$ we have 
$$U^{\det^t}_\ss=U^{\det^t}_\st=\Hom(V,Z)^{\rm epi},$$
 and the quotient $U^{\det^t}_\ss/\GL(Z)$ can be identified with the Grassmannian $\Gr_{k}(V^\smvee)$  (see Example \ref{ExGrass}). Denote by   $E$ the vector bundle  associated with the principal $\GL(Z)$-bundle $U^{\det^t}_\ss\to \Gr_{k}(V^\smvee)$ and the representation $F= \bigoplus_{\lambda\in P(k)} N_\lambda^\smvee\otimes S^\lambda Z^\smvee$, and by $E_i$ the associated bundle with fiber $F_i$. Then
 $$E=\bigoplus_{\lambda\in P(k)} N_\lambda^\smvee\otimes S^\lambda T, $$
where $T$ is the tautological subbundle of $\Gr_k(V^\smvee)$. By the Borel-Weil theorem we get 
$$H\coloneqq H^0(\Gr_k(V^\smvee), {\cal E}^\smvee)=\bigoplus_{\lambda\in P(k)} N_\lambda\otimes S^\lambda V.
$$
The condition $1\leq k<N$ implies  $\mathrm{codim}(U^{\det^t}_\us)\geq 2$, so  
$$H^0(\Gr_{k}(V^\smvee),E_i^\smvee)={\rm Map}_{\GL(Z)}(U^{\det^t}_\ss,F_i^\smvee)={\rm Map}_{\GL(Z)}(U,F_i^\smvee),$$
where ${\rm Map}_{\GL(Z)}$ stands for the set of  $\GL(Z)$-equivariant regular maps. Any $\GL(Z)$-equivariant regular map $U\to F_i^\smvee$ is homogeneous of degree $d_i$. 
 Therefore the data of a section $s\in H^0(\Gr_{k}(V^\smvee),E^\smvee)$ is equivalent to the data of  a system $\sigma=(\sigma_1,\dots,\sigma_r)$  of 
 homogeneous covariants $\sigma_i\in (\C[U]\otimes F_i^\smvee)^{\GL(Z)}$ on $U$ of type $F_i^\smvee$ \cite[p. 9]{KrPr}. This system   can be regarded as a  covariant  of type $F^\smvee$ on $U$ which extends the $\GL(Z)$-equivariant map $\hat s:U^{\det^t}_\ss\to F^\smvee$ associated with $s$.   On the other hand, using formula (\ref{optimal+GR}) it follows that any optimal destabilizing element $\xi\in \Xi^{\det^t}(u)$ of  an unstable  point  $u\in  U_{\us}^{\det^t}$ acts with non-negative weights on tensor  powers $\otimes^{d_i}Z^\smvee$, so also on $F_i$. Therefore  condition (\ref{signs})	in Definition \ref{GLSMDef} is  satisfied, hence the 4-tuple 
$$(\GL(Z),\Hom(V,Z),F=\oplus_{i=1}^rF_i, {\det}^t) $$
 is a GLSM presentation of the pair $(E\to \Gr_k(V^\smvee),s)$.

The obvious right $\GL(V)$-action on $U=\Hom(V,Z)$ satisfies the conditions of Definition \ref{HomGLSMDef}, hence any such GLSM presentation   is homogeneous. We will apply the general results proved in sections \ref{TiltingIH} and \ref{HomGLSMSect} to this class of    GLSM's.

\begin{re}
The cones $C(S^\lambda T)\subset S^\lambda V^\smvee$ are the higher rank varieties first studied by O. Porras \cite{Po}	, and later by J. Weyman \cite[Chapter 7]{We}. We refer to the more general cones 
$$C(\bigoplus_{\lambda\in P(k)} N_\lambda^\smvee\otimes S^\lambda T)\subset \bigoplus_{\lambda\in P(k)} N_\lambda^\smvee \otimes S^\lambda V^\smvee$$
as generalized higher rank varieties.
\end{re}

Let $P(k,N-k)\subset P(k)$ be the set of partitions $(\alpha_1,\cdots,\alpha_k)$ with $\alpha_1\leq N-k$. Kapranov \cite{Ka} has shown that the collection of locally free sheaves $$(S^\alpha {\cal U}^\smvee)_{\alpha\in P(k,N-k)}$$  is a full strongly exceptional collection on $\Gr_k(V^\smvee)$. The Kapranov bundle
$$
{\cal T}_0\coloneqq \bigoplus_{\alpha\in P(k,N-k)} S^\alpha {\cal U}^\smvee
$$
is therefore a tilting bundle on $\Gr_k(V^\smvee)$. Let $E=\bigoplus_{\lambda\in P(k)} N_\lambda^\smvee\otimes S^\lambda T$ be the bundle above, and denote by $\pi: E\to  \Gr_k(V^\smvee)$ its bundle projection. Recall (Corollary \ref{co3} (3)) that $\pi^*({\cal T}_0)$ is a tilting object in $D^b(\Qcoh E)$ provided the following cohomology vanishing holds true:
\begin{equation}\label{eq:*}
H^i\big(\Gr_k(V^\smvee), {\cal T}_0^\smvee\otimes {\cal T}_0\otimes (\bigoplus_{m\geq 0} S^m {\cal E}^\smvee)\big)=0 \ \forall i>0.	
\end{equation}

\begin{thry}\label{th:2-11} Let $E=\bigoplus_{\lambda\in P(k)} N_\lambda^\smvee\otimes S^\lambda T$, let ${\cal T}_0$ be the Kapranov bundle on on $\Gr_k(V^\smvee)$. Suppose that the cohomology vanishing condition (\ref{eq:*}) holds true, and let $s\in H^0(\Gr_k(V^\smvee),{\cal E}^\smvee)$ be a regular section. Then there exists an exact equivalence
$$D^b(\coh Z(s))\simeq D^{\mathrm{gr}}_{\mathrm{sg}}(\Lambda/ s\Lambda).$$	
\end{thry}
\begin{proof}
We verify the assumptions of Theorem \ref{FourthEq}. We know that ${\cal T}_0$ is a tilting bundle on $\Gr_k(V^\smvee)$, and that ${\cal T}=\pi^*({\cal T}_0)$ is a tilting object in $D^b(\Qcoh E)$ provided the vanishing condition  (\ref{eq:*}) holds. Clearly $H=\bigoplus_{\lambda\in P(k)} N_\lambda\otimes S^\lambda V$ generates ${\cal E}^\smvee$, so $E$ is a subbunde of the trivial bundle $\Gr_k(V^\smvee)\times H^\smvee$.
\end{proof}
Note that the cone $C(E)$ is normal and Cohen-Macaulay by Proposition \ref{KempfProp} (1). Moreover,  assuming $\dim(C(E))=\dim(E)$, the Kempf collapsing $\rho:E\to C(E)$ is birational and $C(E)$ has rational singularities  by  Proposition \ref{KempfProp} (2).

\begin{re} In our examples in subsection \ref{sec:2-4-3}, the conditions (i) and (ii) can be found in the literature, or can be checked ``by hand". It is possible to give general sufficient conditions which imply (i) or (ii). For (i) one needs a Borel-Bott-Weil type argument as e.g. in \cite[Proposition 1.4]{BLVdB3}. For (ii) general results can be found in \cite[3.3.5, 3.3.6]{Po} or \cite[7.1.4]{We}.
	
\end{re}

\subsubsection{An algebraic description of \texorpdfstring{$\Lambda$}{str3} for higher rank varieties}
\label{sec:2-4-2}

Throughout this  section $(V,k,\lambda,\sigma)$ denotes a 4-tuple consisting of a complex vector space $V$ of dimension $N$, an integer $k$ with $1\leq k <N$, a partition $\lambda\in P(k)$, and a tensor $\sigma\in S^\lambda V$. Let $s_\sigma\in H^0(\Gr_k(V^\smvee), S^\lambda {\cal U}^\smvee)$ be the section in $S^\lambda {\cal U}^\smvee$ defined by $\sigma$. 

Recall that $\Lambda$ is the graded $\C$-algebra 
$$\Lambda= \bigoplus_{m\geq 0} H^0(\Gr_k(V^\smvee), {\cal T}_0^\smvee\otimes {\cal T}_0\otimes S^m {\cal E}^\smvee),
$$
and the section $s_\sigma$ is an element of the commutative ring $R\subset \Lambda$,
$$R=\bigoplus_{m\geq 0} H^0(\Gr_k(V^\smvee), S^m {\cal E}^\smvee).
$$
In this section we show that in the case of higher rank varieties, $\Lambda$ and the ideal $s\Lambda$ have a purely algebraic description in terms of the initial data $(V,k,\lambda,\sigma)$ provided Theorem \ref{co10} and Theorem \ref{th:2-11} apply.

We have shown in section \ref{Cequiv}, that $\Lambda$ can then be identified with the endomorphism algebra $\End_R(M)$ of the graded module
$$M\coloneqq \bigoplus_{m\geq 0} H^0(\Gr_k(V^\smvee), {\cal T}_0\otimes S^m {\cal E}^\smvee).
$$
Therefore we need an algebraic description  of $R$ as a graded ring, an identification of the element $s\in R_1$, and a description of $M$ as a graded $R$-module.

Let $S^\lambda V$ be the Schur representation of $\GL(V)$ defined by $\lambda\in P(k)$, and denote by $S=S^\bullet (S^\lambda V)=\bigoplus_{m\geq 0} S^m(S^\lambda V)$ the symmetric algebra  of $S^\lambda V$ endowed with its natural grading.  

\begin{pr} (Porras) \label{prop:2-12} The ideal $I_k$ defining the higher rank variety $C(S^\lambda T)$ consists of all representations $S^\mu V\subset S$, $\mu=(\mu_1,\cdots,\mu_t)$, with $t>k$. The graded ring $R$ is isomorphic to the graded quotient ring $S/I_k$.	
\end{pr}
\begin{proof}
The first assertion is \cite[3.3.2]{Po}, the second is \cite[3.3.3]{Po}.	
\end{proof}

Recall that  $s_\sigma\in H^0(\Gr_k(V^\smvee), S^\lambda {\cal U}^\smvee)=R_1$ is given by the tensor $\sigma\in S^\lambda V$. Then $\bar \sigma\in S/I_k$ corresponds to $s_\sigma\in R$, hence $s_\sigma\Lambda$ is the ideal generated by $\bar \sigma\in R\subset\Lambda$.

In order to describe $M$ as a graded $R$-module we use the map
$$\varphi: S^{\lambda/1}V\otimes S\to V^\smvee\otimes S\langle 1\rangle 
$$
of free graded $S$-modules defined in \cite[3.2.1]{Po}.

Consider the Kapranov bundle 
$${\cal T}_0=\bigoplus_{\alpha\in P(k,N-k)} S^\alpha {\cal U}^\smvee, $$
and the graded $R$-module
$$M_\alpha=H^0(S^\lambda T, \pi^* S^\alpha {\cal U}^\smvee)=H^0(\Gr_k(V^\smvee), S^\alpha {\cal U}^\smvee\otimes S^\bullet(S^\lambda {\cal U}^\smvee)).
$$
We have $M=\bigoplus_{\alpha\in P(k,N-k)} M_\alpha$, so that it suffices to describe each $M_\alpha$ as a graded $R$-module. 
\begin{pr}\label{prop:2-13} $M_\alpha$ is isomorphic to the image of the following morphism of free  graded $R$-modules:
$$S^\alpha(\varphi^\smvee)\otimes R: S^\alpha V\otimes R\langle-1\rangle\to S^\alpha(S^{\lambda/1} V^\smvee)\otimes R.
$$
	
\end{pr}
\begin{proof}
The map $\varphi: S^{\lambda/1} V\otimes S\to V^\smvee \otimes S\langle 1\rangle$	 corresponds to a map of trivial vector bundles over $S^\lambda V^\smvee$:
$$
\begin{tikzcd}
S^{\lambda/1}\underline{V}  \ar[dr] \ar[rr, "\tilde\varphi"] & &\underline{V}^\smvee \ar[dl]\\
&S^\lambda V^\smvee &
\end{tikzcd}.
$$
The pull-back of its dual $\tilde \varphi^\smvee$ via the Kempf collapsing $\rho:S^\lambda T\to C(S^\lambda T)\subset S^\lambda V^\smvee$ induces a map 
$$\rho^* \tilde \varphi^\smvee: S^\lambda T\times V\to S^\lambda T\times S^{\lambda/1} V^\smvee
$$
which factorizes over $\pi^* {\cal U}^\smvee$:
$$
\begin{tikzcd}
S^{\lambda}T \times V  \ar[dr, "\varepsilon"'] \ar[rr, "\rho^* \tilde \varphi^\smvee"] & &S^{\lambda}T\times S^{\lambda/1} V^\smvee \\
&\pi^* {\cal U}^\smvee \ar[ur, "\iota"'] &
\end{tikzcd} 
$$
Here $\varepsilon$ is a vector bundle epimorphism, and $\iota$ is a sheaf monomorphism. Applying the Schur functor $S^\alpha$ yields an epi-mono factorization:
\begin{equation}\label{epi-mono}
\begin{tikzcd}
S^\alpha\underline{V}\ar[r, two heads] &\pi^* S^\alpha {\cal U}^\smvee \ar[r, hook] &  S^\alpha(S^{\lambda/1}\underline{V}^\smvee).
\end{tikzcd}
\end{equation}
Now the argument of \cite[3.5]{BLVdB2} can be used: it suffices to show that the composition 

\begin{equation}\label{eq:**}
\begin{tikzcd}
S^\alpha\underline{V}\otimes S^\bullet(S^\lambda\underline{V})\ar[r] & S^\alpha\underline{V}\otimes S^\bullet (S^\lambda {\cal U}^\smvee) \ar[r]& S^\alpha{\cal U}^\smvee\otimes S^\bullet (S^\lambda {\cal U}^\smvee)
\end{tikzcd}
\end{equation}
remains surjective after taking sections on $\Gr_k(V^\smvee)$. For this the filtration argument in \cite[3.5]{BLVdB2} applies verbatim.
\end{proof}
We can now state our final result:
\begin{thry}\label{th:2-14}
Let $(V,k,\lambda,\sigma)$ be the initial data as above, such that 
$$s_\sigma\in H^0(\Gr_k(V^\smvee), S^\lambda {\cal U}^\smvee)$$
 is a regular section with zero locus $Z(s_\sigma)\subset \Gr_k(V^\smvee)$. Let $I_k\subset S$ be the graded ideal consisting of all representations $S^{(\mu_1,\cdots,\mu_t)} V\subset S$ with $t>k$, and let $\bar \sigma\in S/I_k$ be the image of $\sigma$.  Let $\varphi: S^{\lambda/1} V\otimes S\to V^\smvee\otimes S\langle 1\rangle $ be Porras' map. If one of the conditions in (1) or (2) of Theorem \ref{co10}, and both of the conditions (i) and (ii) of Theorem \ref{th:2-11}	 are satisfied, then the bounded derived category of $Z(s_\sigma)$ has the following purely algebraic description in terms of the initial data $(V,k,\lambda,\sigma)$:
$$D^b(\coh Z(s_\sigma))\simeq D^{\mathrm{gr}}_{\mathrm{sg}}\big(\End_{S/I_k}\big(\bigoplus_{\alpha\in P(k,N-k)}\mathrm{Im}(S^\alpha(\varphi^\smvee)\otimes (S/I_k))\big)\big/\langle\bar\sigma\rangle\big).
$$

\end{thry}

\subsubsection{Some concrete examples}
\label{sec:2-4-3}

\begin{enumerate}[1.]
\item {\it Complete intersections.}  In this case the algebra $\Lambda$ can be described explicitly, and the conditions (i) and (ii) of Theorem \ref{th:2-11} can be easily checked. Choose $k\coloneqq \dim(Z)=1$, $F_i=S^{d_i} Z^\smvee$. In this case we have $\GL(Z)\simeq\C^*$, $\Gr_k(V^\smvee)=\P(V^\smvee)$, and
$$E=|\bigoplus_{i=1}^r {\cal O}_{V^\smvee}(-d_i)|, \  H^0(\P(V^\smvee),E^\smvee)=\bigoplus_{i=1}^r S^{d_i} V.
$$
The zero locus $Z(s)\subset \P(V^\smvee)$ of a section  $s\in\Gamma(\P(V^\smvee),E^\smvee)$ is the complete intersection 
$$\bigcap_{i=1}^r Z_h(\sigma_i)\subset \P(V^\smvee),$$
where 
$$\sigma=(\sigma_1,\dots,\sigma_r)\in \bigoplus_{i=1}^r S^{d_i}V$$
 is the system of polynomials defining $s$. The corresponding GLSM presentation is the 4-tuple 
 $$(\GL(Z)\simeq\C^*,\,\Hom(V,Z),\,\bigoplus_{i=1}^r S^{d_i} Z^\smvee,\,{\det}^{t})$$ with $t\in\N_{>0}$. 

 The map 
 $$ \vartheta:E\to H^0(\P(V^\smvee),E^\smvee)^\smvee= \bigoplus_i S^{d_i} V^\smvee$$
  acts as follows: The fibre $E_{l}$ over a point $l \in \P(V^\smvee)$  is the product $\bigtimes_{i=1}^r l^{\otimes d_i}$, and its image $ \vartheta(E_l)$ in $\bigoplus_i S^{d_i} V^\smvee $ is the $r$-dimensional subspace $\bigoplus_{i=1}^r l^{\otimes d_i}\subset \bigoplus_i S^{d_i} V^\smvee$ spanned by the lines $l^{\otimes d_i}\subset S^{d_i} V$. Therefore $\P(\vartheta)$ is an embedding,  and the  corresponding projective variety 
 $$S(E)\coloneqq \P(\vartheta)(\P(E))=\union_{l\in \P(V^\smvee)}\P(\bigoplus_{i=1}^r l^{\otimes d_i})\subset \P(\bigoplus_{i=1}^r S^{d_i} V^\smvee)$$
 is smooth. According to Corollary \ref{PropertiesS(E)} the projective variety $S(E)$ is projectively normal, arithmetically  Cohen-Macaulay, and the vertex of its affine cone $C(E)$ (its only singularity)  is a  rational singularity.  Moreover,  if $N=\sum_{i=1}^r d_i$  then $S(E)$ is also arithmetically Gorenstein.
  
The graded ring $R=\C[C(E)]=\C[\Hom(V,Z)\oplus(\oplus_{i=1}^r S^{d_i}Z^\smvee)]^{\GL(Z)} $ is
 $$R=\bigoplus_{m\geq 0} H^0\big(\P(V^\smvee), S^m(\bigoplus_{i=1}^r {\cal O}_{V^\smvee}(d_i))\big)=\bigoplus_{m\geq 0} H^0\big(\P(V^\smvee), \bigoplus_{\substack{k\in \N^r\\|k|=m}} {\cal O}_{V^\smvee}(\sum_{i=1}^r k_id_i)\big)
 $$
 $$=\bigoplus_{m\geq 0}\big(\bigoplus_{\substack{k\in \N^r\\|k|=m}}S^{\sum_{i=1}^r k_id_i}(V)\big),
 $$
 and its multiplication  is given by the obvious bilinear maps
 $$S^{\sum_{i=1}^r k_id_i}(V)\times S^{\sum_{i=1}^r l_id_i}(V)\to S^{\sum_{i=1}^r (k_i+l_i)d_i}(V).
 $$

 The locally free sheaf   
 $${\cal T}_0\coloneqq \bigoplus_{i=0}^{N-1}{\cal O}_{V^\smvee}(i)  
 $$
generates $D^b(\coh\P(V^\smvee))$ classically, and satisfies the condition of Corollary \ref{co3} (3), so ${\cal T}\coloneqq \pi^*({\cal T}_0)$ is a tilting bundle on $E$ and  Theorem \ref{FourthEq} applies, giving a purely algebraic interpretation of the derived category of the complete intersection $Z(s)=\cap_{i=1}^r Z_h(\sigma_i)$:
 
  The graded ring $\Lambda=\End_E({\cal T})$ intervening in this algebraic interpretation  is 
 $$\Lambda=\bigoplus_{m\geq 0}H^0(\P(V^\smvee),{\cal T}_0^\smvee\otimes {\cal T}_0\otimes S^m({\cal E}^\smvee))=\bigoplus_{m\geq 0}\bigoplus_{\substack{0\leq a\leq N-1\\ 0\leq b\leq N-1}} \bigoplus_{\substack{k\in \N^r\\ |k|=m}}S^{\sum_{i=1}^r k_i d_i+b-a}(V)
 $$
 and its multiplication is induced by the obvious bilinear maps
 $$S^{\sum_{i=1}^r k_i d_i+b-a}(V)\times  S^{\sum_{i=1}^r l_i d_i+c-b}(V)\to S^{\sum_{i=1}^r (k_i+l_i) d_i+c-a}(V).
 $$
 Note also that, by Corollary \ref{co9}, in the case $N=\sum_{i=1}^r d_i$ the ring $\Lambda$ is a crepant resolution of $R$. 
 
 \vspace{3mm}

 Interesting special cases of this family of  Abelian GLSM presentations  associated with complete intersections  have been studied by several authors. The case $N=5$, $r=1$, $d_1=5$ reproduces Witten's original GLSM \cite{Wi}. In \cite{CDHPS} the authors study the cases   
$$(d_1,\dots,d_r)\in\big\{(2,2), (2,2,2),(2,2,2,2),(3),(3,3)\big\}$$
 for different  values of $N=\dim(V)$.  \\
 
 In the following two cases we use Theorem \ref{th:2-11}, so it suffices to verify conditions (i) and (ii) of Theorem \ref{th:2-11}. 
 \\  
 
 \item {\it Isotropic orthogonal Grassmannians.} Let $q\in S^2V$ be a non-degenerate quadratic form on $V^\smvee$. Let $1\leq k\leq\frac{N}{2}$. The isotropic Grassmannian $\Gr_k^q(V^\smvee)\subset \Gr_k(V^\smvee)$ is the submanifold of $k$-dimensional isotropic subspaces of $(V^\smvee,q)$:
$$\Gr_k^q(V^\smvee)\coloneqq \{K\subset V^\smvee|\ K\hbox{ linear subspace of dimension }k,\ \resto{q}{K}\equiv 0\}.$$
The form $q$ defines a section $s_q\in \Gamma(\Gr_k(V^\smvee),S^2 T^\smvee)$ which is transversal to the zero section, and whose zero locus is  $\Gr_k^q(V^\smvee)$.  Therefore 
$$\dim(\Gr_k^q(V^\smvee))=k(N-k)-\binom{k+1}{2},$$
and the 4-tuple $(\GL(Z),\Hom(V,Z),S^2 Z^\smvee,\det^t)$ (with  $t\in \N_{>0}$) is a GLSM presentation  of   $(\extp^2 T\to  \Gr_k(V^\smvee),s_q)$.

Note that $\Gr_k^q(V^\smvee)$ comes with an obvious action of the group $\SO(V^\smvee,q)$. This action is transitive unless $N=2k$ \cite[section 4]{BKT}. In the latter case $\Gr_k^q(V^\smvee)$ has two connected components $\Gr_k^q(V^\smvee)_\pm$, and $\SO(V^\smvee,q)$ acts transitively on each component. One has isomorphisms   $\Gr_k^q(V^\smvee)_\pm\simeq \Gr_{k-1}^{q_U}(U)$, where $U\subset V^\smvee$ is a  general hyperplane,  and $q_U$ is the restriction of $q$ to $U$, given by intersecting with $U$.  The conditions (i) and (ii) of Theorem \ref{th:2-11} are in \cite[Theorem B]{WZ}.
\\

\item {\it Isotropic symplectic Grassmannians.}. Let $V$ be a complex vector space of even dimension $N=2n$, and let $\omega\in \extp^2V$ be a symplectic form on $V^\smvee$. For even $k$ with $2\leq k\leq n$, the isotropic Grassmannian $\Gr_k^\omega(V^\smvee)\subset \Gr_k(V^\smvee)$  is the submanifold of $k$-dimensional isotropic subspaces of $(V^\smvee,\omega)$:
$$\Gr_k^\omega(V^\smvee)\coloneqq \{K\subset V^\smvee|\ K\hbox{ linear subspace of dimension }k,\ \resto{\omega}{K}\equiv 0\}.$$
Denoting by $T$ the tautological $k$-bundle of $\Gr_k(V^\smvee)$, the form $\omega$ defines a section $s_\omega\in \Gamma(\Gr_k(V^\smvee),\extp^2 T^\smvee)$ which is transversal to the zero section, and whose zero locus is  $\Gr_k^\omega(V^\smvee)$.  Therefore 
$$\dim(\Gr_k^\omega(V^\smvee))=k(N-k)-\binom{k}{2},$$
and the 4-tuple $(\GL(Z),\Hom(V,Z),\extp^2 Z^\smvee,\det^t)$ (with  $t\in \N_{>0}$) is a GLSM presentation  of   $(\extp^2 T \to  \Gr_k(V^\smvee),s_\omega)$. Note that $\Gr_k^\omega(V^\smvee)$ comes with a transitive action of the group $\mathrm{Sp}(V^\smvee,\omega)$ \cite[section 4]{BKT}. The conditions (i) and (ii) of  Theorem \ref{th:2-11} are in \cite[Theorem C]{WZ}.\\

In our final example we can apply Theorem \ref{th:2-14}. We obtain a graded crepant nc  resolution of the invariant ring $R$, and a purely algebraic description of $D^b(\coh(Z(s_\sigma)))$.\\
\item {\it Beauville-Donagi IHS 4-folds. } Let $V$ be a complex vector space of dimension $N$, and $\sigma\in S^3V$. Choosing $k=\dim(Z)=2$, $r=1$, and $F_1=S^3 Z^\smvee$, we get the bundle $S^3 T$ on $\Gr_2(V^\smvee)$, and a section $s_\sigma\in H^0(\Gr_2(V^\smvee), S^3{\cal U}^\smvee)$. The zero locus $Z(s_\sigma)\subset \Gr_2(V^\smvee)$  is the Fano variety of lines in the cubic hypersurface  $Z_h(\sigma)\subset P(V^\smvee)$ defined by $\sigma\in S^3 V=H^0(\P(V^\smvee),{\cal O}_{V^\smvee}(3))$. For $N=6$ and $\sigma$ general, $Z(s_\sigma)$ is a Beauville-Donagi IHS 4-fold \cite{BD}. Note that the symmetric algebra $S=S^\bullet(S^3V)$ is not multiplicity free. Condition (i) of Theorem \ref{th:2-11} is in \cite[section 2]{Kan}, and we have verified the vanishing condition (ii) by a Borel-Bott-Weil computation.

In order to see that condition (2)(i) of Theorem \ref{co10} is satisfied, we prove a general lemma which computes the codimension of the pre-image of the exceptional locus of the Kempf collapsing in this case.

Let $V$ be a complex vector space of dimension $N$, and $d$, $k$ be positive integers with $d\geq 2$, $k<N$. Let $T_k$  be the tautological subbundle of $\Gr_k(V^\smvee)$, and $\rho: S^dT_k\to C(S^d T_k)\subset S^d V^\smvee$ the Kempf collapsing.  The cone $C(S^d T_k)$ can be described as follows: 

For an element $q\in S^dV^\smvee$ put
$$L_q\coloneqq \bigcap_{\substack{L\subset V^\smvee\\ q\in S^d L}} L\ ,\ \ \rk(q)\coloneqq \dim(L_q).
$$ 
It is easy to see that $L_q$ is the minimal subspace of $V^\smvee$ whose $d$-symmetric power contains $q$, and this definition of $\rk(q)$ agrees with the algebraic definition, i.e. with the rank of the linear map $V\to S^{d-1} V^\smvee$ associated with $q$. 
One has
$$C(S^d T_k)=S^dV^\smvee_{\leq k}\coloneqq \{q\in S^d V^\smvee|\ \rk(q)\leq k\},
$$
so  $C(S^d T_k)$ coincides with the catalecticant variety associated with the triple $(V,d,k)$.   In  \cite[Theorem 2.2]{Kan} Kanev shows that this variety is irreducible, normal, Cohen-Macauley of dimension $\binom{k+d-1}{d}+k(N-k)$, with rational singularities along $\mathrm{Sing}(S^dV^\smvee_{\leq k})= S^dV^\smvee_{\leq k-1}$. For $0\leq r\leq N$ put
$$S^dV^\smvee_r\coloneqq S^dV^\smvee_{\leq r}\,\setminus\, S^dV^\smvee_{\leq r-1}.
$$
For a point $q\in S^dV^\smvee_r$,   $L_q$ is the unique $r$-dimensional subspace of $V^\smvee$ whose $d$-symmetric power contains $q$.  The map $\gamma_r: S^dV^\smvee_{r}\to \Gr_r(V^\smvee)$ given by $\gamma_r(q)=L_q$ is regular; its graph is the intersection
$$\Gamma_r=[S^dV^\smvee_{r}\times \Gr_r(V^\smvee)]\cap S^dT_r.
$$
Let now $r\leq k$. The restriction
$$\resto{\rho}{\rho^{-1}(S^dV^\smvee_{r})}: \rho^{-1}(S^dV^\smvee_{r})\to S^dV^\smvee_{r}
$$
is a fiber bundle. Its fiber over a point $q\in S^dV^\smvee_{r}$ can be identified with 
$$\{ U\in \Gr_k(V^\vee)|\ L_q\subset  U\}\simeq\Gr_{k-r}(V^\smvee/L_q). 
$$

This proves: 
\begin{lm}
With the notations above the following holds:
\begin{enumerate}[(1)]
\item $\mathrm{Exc}(\rho)=S^dV^\smvee_{\leq k-1}$.
\item For $0\leq r\leq k$ one has 
$$\dim(\rho^{-1}(S^dV^\smvee_{r}))=\binom{r+d-1}{d}+r(N-r)+ (k-r)(N-k),$$
$$\mathrm{codim}_{S^dT_k}(\rho^{-1}(S^dV^\smvee_{r}))=\binom{k+d-1}{d}-\binom{r+d-1}{d}-r(k-r).
$$
\item One has
 $$\mathrm{codim}_{S^dT_k}(\rho^{-1}(\mathrm{Exc}(\rho)))=\underset{0\leq r < k}{\min}\bigg\{\binom{k+d-1}{d}-\binom{r+d-1}{d}-r(k-r)\bigg\}.
 $$
\end{enumerate}	
\end{lm}
\begin{re}
In the case $d=2$, the codimension of $\rho^{-1}(\mathrm{Exc}(\rho))$ in $S^dT_k$ is always 1.	
\end{re}

\end{enumerate}

\end{document}